\documentclass[a4paper]{amsart}
\usepackage[T1]{fontenc}
\usepackage[utf8]{inputenc}
\usepackage{amscd,amssymb,amsmath,amsthm, amsfonts}
\usepackage{epsfig}
\usepackage{mathrsfs}
\usepackage{hyperref}
\hypersetup{colorlinks=true,pageanchor=false,linkcolor=blue,citecolor=red}
\usepackage{enumitem}
\usepackage{comment}
\usepackage{marginnote}
\usepackage{stackrel}
\usepackage{tabu}

\newtheorem{theorem}{Theorem}[section]
\newtheorem{corollary}[theorem]{Corollary}
\newtheorem{proposition}[theorem]{Proposition}
\newtheorem{lemma}[theorem]{Lemma}

\newtheorem{theoremi}{Theorem}
\newtheorem{corollaryi}[theoremi]{Corollary}

\theoremstyle{definition}
\newtheorem{definition}[theorem]{Definition}

\newtheorem{notation}[theorem]{Notation}

\theoremstyle{remark}
\newtheorem{remark}[theorem]{Remark}
\newtheorem{fact}{Fact}[theorem]
\newtheorem{conjecture}[theorem]{Conjecture}
\newtheorem{conjecturei}[theoremi]{Conjecture}

\DeclareUnicodeCharacter{2297}{{\otimes}}
\DeclareMathSymbol{\widetildesym}{\mathord}{largesymbols}{"65}
\DeclareUnicodeCharacter{0393}{\Gamma} 
\DeclareUnicodeCharacter{0394}{\Delta} 
\DeclareUnicodeCharacter{0398}{\Theta} 
\DeclareUnicodeCharacter{039B}{\Lambda} 
\DeclareUnicodeCharacter{039E}{\Xi} 
\DeclareUnicodeCharacter{03A0}{\Pi} 
\DeclareUnicodeCharacter{03A3}{\Sigma} 
\DeclareUnicodeCharacter{03A5}{\Upsilon} 
\DeclareUnicodeCharacter{03A6}{\Phi} 
\DeclareUnicodeCharacter{03A8}{\Psi} 
\DeclareUnicodeCharacter{03A9}{\Omega} 

\DeclareUnicodeCharacter{03B1}{\alpha} 
\DeclareUnicodeCharacter{03B2}{\beta} 
\DeclareUnicodeCharacter{03B3}{\gamma} 
\DeclareUnicodeCharacter{03B4}{\delta} 
\DeclareUnicodeCharacter{03B5}{\varepsilon} 
\DeclareUnicodeCharacter{03B6}{\zeta} 
\DeclareUnicodeCharacter{03B7}{\eta} 
\DeclareUnicodeCharacter{03B8}{\theta} 
\DeclareUnicodeCharacter{03B9}{\iota} 
\DeclareUnicodeCharacter{03BA}{\kappa} 
\DeclareUnicodeCharacter{03BB}{\lambda} 
\DeclareUnicodeCharacter{03BC}{\mu} 
\DeclareUnicodeCharacter{03BD}{\nu} 
\DeclareUnicodeCharacter{03BE}{\xi} 
\DeclareUnicodeCharacter{03C0}{\pi} 
\DeclareUnicodeCharacter{03C1}{\rho} 
\DeclareUnicodeCharacter{03C2}{\varsigma} 
\DeclareUnicodeCharacter{03C3}{\sigma} 
\DeclareUnicodeCharacter{03C4}{\tau} 
\DeclareUnicodeCharacter{03C5}{\upsilon} 
\DeclareUnicodeCharacter{03C6}{\varphi} 
\DeclareUnicodeCharacter{03C7}{\chi} 
\DeclareUnicodeCharacter{03C8}{\psi} 
\DeclareUnicodeCharacter{03C9}{\omega} 

\DeclareUnicodeCharacter{03D1}{\thetasymbol} 
\DeclareUnicodeCharacter{03D5}{\phisymbol} 
\DeclareUnicodeCharacter{03D6}{\pisymbol} 
\DeclareUnicodeCharacter{03DD}{\digamma} 
\DeclareUnicodeCharacter{03F1}{\rhosymbol} 
\DeclareUnicodeCharacter{03F5}{\epsilonsymbol} 

\newcommand{\tcoev}{\stackrel{\longrightarrow}{\operatorname{coev}}}
\newcommand{\tev}{\stackrel{\longrightarrow}{\operatorname{ev}}}
\newcommand{\ev}{\stackrel{\longleftarrow}{\operatorname{ev}}}
\newcommand{\coev}{\stackrel{\longleftarrow}{\operatorname{coev}}}
\newcommand{\e}{\operatorname{e}}
\newcommand{\sch}{\operatorname{sch}} 
\newcommand{\kk}{\Bbbk}
\newcommand{\wb}{\overline}
\newcommand{\bp}[1]{{\left(#1\right)}}
\newcommand{\catb}{\mathscr{B}} 
\newcommand{\cat}{\mathscr{C}}
\newcommand{\catc}{\mathscr{C}}
\newcommand{\catp}{\mathscr{D}^{\wp}}
\newcommand{\catn}{{\mathscr{D}}}
\newcommand{\catpb}{\bar{\mathscr{D}}^{\wp}}
\newcommand{\catpz}{\mathscr{D}_{\bar{z}}^{\wp}}
 \newcommand{\catN}{\mathscr{D}^{\aleph}}
 \newcommand{\Gp}{{\Gr^{\wp}}}
\newcommand{\brk}[1]{{{\left\langle{#1}\right\rangle}}}
\newcommand{\C}{\ensuremath{\mathbb{C}}}
\newcommand{\Z}{\ensuremath{\mathbb{Z}}}
\newcommand{\Zt}{\ensuremath{\mathsf{Z}}}
\newcommand{\Ztt}{\ensuremath{\mathsf{Z}^{\wp}}}

\newcommand{\N}{\ensuremath{\mathbb{N}}}
\newcommand{\h}{\ensuremath{\mathfrak{h}}}
\newcommand{\HR}{\ensuremath{{\mathcal H}}}
\newcommand{\Rm}{\ensuremath{{\mathcal R}}}
\newcommand{\g}{\ensuremath{\mathfrak{g}}}
\newcommand{\UH}{\ensuremath{U_ξ^H(\g)}}
\newcommand{\borel}{\ensuremath{\mathfrak{b}}}
\newcommand{\nil}{\ensuremath{\mathfrak{n}}}
\newcommand{\sll}{\ensuremath{\mathfrak{sl}}}
\newcommand{\gl}{\ensuremath{\mathfrak{gl}}}
\newcommand{\slmn}{\sll(m|n)}
\newcommand{\End}{\operatorname{End}}
\newcommand{\Hom}{\operatorname{Hom}}
\newcommand{\unit}{\ensuremath{\mathbb{I}}}
\newcommand{\Id}{\operatorname{Id}}
\newcommand{\qn}[1]{{\left\{#1\right\}}}
\newcommand{\qN}[1]{{\left[#1\right]}}
\newcommand{\qd}{\operatorname{\mathsf{d}}}

\newcommand{\K}{\ensuremath{\mathbb{K}}}
\newcommand{\ptr}{\operatorname{ptr}}
\newcommand{\str}{\operatorname{str}}

\newcommand{\Proj}{{\operatorname{Proj}}}
\newcommand{\calB}{\mathcal B}
\renewcommand{\leq}{\leqslant}

\newcommand{\ve}{\varepsilon}
\newcommand{\vp}{\varphi}

\newcommand{\Ux}{{U_ξ}}
\newcommand{\frakg}{\mathfrak g}
\newcommand{\TT}{{\ddot\C}}

\newcommand{\catd}{\mathscr{D}}
\newcommand{\Gr}{{\mathsf G}}
\newcommand{\XX}{\ensuremath{\mathsf{X}}}
\newcommand{\Xp}{\ensuremath{\XX^\wp}}
\newcommand{\Tp}{\ensuremath{\Theta^\aleph}}

\newcommand{\mt}{\operatorname{\mathsf{t}}}
\newcommand{\Negl}{\operatorname{Negl}}
\newcommand{\ideal}{\mathcal{I}} 
\newcommand{\idealp}{\ideal^\wp} 
\newcommand{\idealpb}{\bar\ideal^\wp} 
\newcommand{\qdp}{\qd^\wp} 
\newcommand{\mtp}{\mt^\wp} 
\newcommand{\p}[1]{\ensuremath{\bar {#1}}}
\newcommand{\roots}{\Delta}
\newcommand{\rk}{r}
\newcommand{\w}[2]{\lambda_#1^#2}

\newcommand{\epsh}[2]
         {\begin{array}{c} \hspace{-1.3mm}
        \raisebox{-4pt}{\epsfig{figure=#1,height=#2}}
        \hspace{-1.9mm}\end{array}}

\newcounter{exo} \newcounter{numexercice}
\renewcommand{\theexo}{\arabic{exo}}

\newcommand{\ms}[1]{\small \ensuremath{#1}}

\numberwithin{equation}{section}
  
\begin{document}

\raggedbottom

\title[Relative (pre)-modular categories from $sl(m|n)$]{Relative (pre)-modular categories from special linear Lie superalgebras.}
\author[CA Anghel]{Cristina Ana-Maria Anghel}
\address{Mathematical Institute, University of Oxford, Oxford, United Kingdom} \email{palmeranghel@maths.ox.ac.uk} 
\author[N Geer]{Nathan Geer}
\address{Utah State University, Department of Mathematics and Statistics, Logan UT 84341, USA}
\email{nathan.geer@usu.edu}
\author[B Patureau]{Bertrand Patureau-Mirand}
\address{Universit\'e Bretagne Sud, Laboratoire de Math\'ematiques de Bretagne Atlantique, UMR CNRS 6205, Campus de Tohannic, BP 573     
F-56017 Vannes, France}
\email{bertrand.patureau@univ-ubs.fr}
 
\maketitle
\setcounter{tocdepth}{3}

\date{\today}
\begin{abstract} We examine two different m-traces in the category of
  representations over the quantum Lie superalgebra associated to
  $\mathfrak{sl}(m|n)$ at root of unity.  The first m-trace is on the
  ideal of projective modules and leads to new Extended Topological
  Quantum Field Theories.  The second m-trace is on the ideal of
  perturbative typical modules.  We consider the quotient with respect
  to negligible morphisms coming from this m-trace and show that in
  the case of $\mathfrak{sl}(2|1)$ this quotient leads to 3-manifolds
  invariants.  We conjecture that the quotient category of
  perturbatives over quantum $\mathfrak{sl}(m|n)$ leads to 3-manifold
  invariants and more generally ETQFTs.
\end{abstract}
\date{\today}
\tableofcontents

\section{Introduction}
  Quantum traces and the corresponding concept of  quantum dimensions are a key tool in applications to low-dimensional topology, representation theory, and other fields.  When the category is not semi-simple, the quantum traces vanish, and these constructions become trivial.  The concept  of \emph{modified traces} (or \emph{m-traces}, for short) are non-trivial replacements for trace functions on non-semi-simple ribbon and, more generally, pivotal categories (e.g. see \cite{GKP1,GPV,GKP3}).  The study and underpinnings of m-traces leads to new, interesting quantum invariants, for example:  1) the renormalized link invariants of \cite{GPT}, 2)  
 the non-semisimple version of the Witten-Reshetikhin-Turaev 3-manifold invariants of \cite{CGP14} associated to  \emph{non-degenerate relative pre-modular categories} and 3) the Extended Topological Quantum Field Theories (ETQFTs) of De Renzi in \cite{D17} associated to  \emph{relative modular categories}.  For the definitions of (non-degenerate) relative (pre-)modular categories, see Subsection \ref{SS:RelModDef}.

 Loosely speaking, the definition of an m-trace is as follows.  Let
 $\cat$ be a pivotal category and $\ideal$ be an ideal (a full
 subcategory of $\cat$ which is closed under 
 retracts
 and absorbing for tensor product%
 ).  Let $\kk=\End_\cat(\unit)$ be the ground ring of $\cat$
 where $\unit$ is the unit in $\cat$.  A m-trace is a family of
 $\kk$-linear functions,
 $\{\mt_V:\End_\cat( V)\rightarrow \kk \}_{V \in \ideal},$ where $V$
 runs over all objects of $\ideal$, and such that certain partial
 trace and cyclicity properties hold.

When $q=\e^h$ is a formal parameter, the category $U_q\slmn$-mod of
finite dimensional modules over the quantum Lie superalgebra
$U_q\slmn$ is not
semi-simple and the quantum trace of a
generic module vanishes.  This vanishing makes the associated standard
Reshetikhin-Turaev link invariant zero.  The non-semi-simplicity also
leads to a hierarchy of ideals and associated m-traces (see for
example the generalized Kac-Wakimoto conjecture proven in \cite{Se}).
In particular, there is an m-trace in $U_q\slmn$-mod on the smallest
ideal of typical modules that was used in \cite{GP1,GP2} to construct
re-normalized Reshetikhin-Turaev link invariants.  These invariants recover the
celebrated sequence of Kashaev's invariants \cite{Kv} as well as
Links-Gould invariants \cite{LG} and they are conjectured to contain the
non-semisimple sequence of ADO polynomials \cite{ADO,GP3}. 

In the standard Reshetikhin-Turaev theory \cite{RT2} to construct a
3-manifold invariant from a quantum link invariant one must set the
quantum parameter to a root of unity.  Here we follow this idea.  The
complexity of the ideals and m-traces in $U_{ξ}\slmn$-mod increases
when $q={ξ}$ is specialized to a root of unity.  Also, the category
$U_{ξ}\slmn$-mod is no longer braided. 

 Our first main result is to
slightly enlarge the algebra $U_{ξ}\slmn$ to the unrolled quantum
super group $U^{H}_{ξ}\slmn$ and prove that the category $\catd$ of
weight modules over $U^{H}_{ξ}\slmn$ is 
ribbon,
see corollaries \ref{C:BraidedCat} and \ref{C:D-ribbon}.
The category $\catd$ 
contains specializations of $U_q\slmn$-modules, which we call 
\emph{perturbative modules}.  Also, after specializing, new projective modules appear with arbitrary complex
weights.  In this paper we consider the following two ideals and corresponding m-traces:
\begin{enumerate}
\item the ideal $\ideal$ of projective modules, 
\item the ideal $\idealp$ of perturbative modules given in Theorem \ref{C:22}.
\end{enumerate}
Let us discuss these two situations in more detail.
The ideal $\ideal$ has a structure very similar to
the case of the nilpotent projective modules over a quantum simple Lie algebra at a root of unity (see also \cite{DGP2,GHP}).  In particular, in Section \ref{S:ProjMod}, the m-trace on the ideal $\ideal$ is one of the main ingredients used to prove the following theorem (in the text below this is Theorem \ref{THEOREM1b}):    
\begin{theoremi}[The projective module case]\label{THEOREM1}
$\catn$ is a relative modular category.
\end{theoremi}
Using the technology of \cite{CGP14},
this produces 3-manifolds invariants and graded ETQFTs as in
\cite{BCGP,D17}:
\begin{corollaryi}[ETQFTs]
There is an Extended Topological Quantum Field Theory associated to $\catn$.
\end{corollaryi}

Next let us describe the case of the ideal of perturbative modules.   Perturbative modules can be seen has as
deformations of modules over the classical Lie superalgebra.  
Loosely speaking, the standard way to obtain a modular category from a quantum group is: 1) specialize $q$ to a root of unity; this forces some modules to have zero quantum dimension, 2) quotient the space of morphisms by \emph{negligible morphisms} of modules with zero quantum trace, 3) show that the resulting category is finite and semi-simple.  Here we generalize this technique to the context of Lie superalgebras: take a quotient by negligible morphisms corresponding to the \emph{m-trace}.  
 In particular, let  $\catp$ be the category of {\em perturbative modules} over $U_{ξ}^{H}(\slmn)$, see Definition \ref{D:PertMod}.  Let $\catpb$ be the full subcategory of $\catp$ generated by what we call standard modules, see Definition \ref{D:SubcatPB}.  Let $\catN$ be the quotient of $\catpb$ by the negligible morphisms corresponding to the m-trace on $\idealp$.  
In Section~\ref{S:perturbative} we prove (in text below this is Theorem \ref{THEOREM2b}):

\begin{theoremi}[Relative pre-modular category from perturbative modules]\label{THEOREM2}
The category $\catN$ is a relative pre-modular category.
\end{theoremi}

The kernel of the m-trace on $\idealp$ contains the ideal of
projective modules $\ideal$.  So the two construction of this paper
are really orthogonal.  We have the following two conjectures (the text below these are 
Conjectures \ref{conj:non-degenerateb} and \ref{conj:modularb}).
\begin{conjecturei}\label{conj:non-degenerate}
  The relative pre-modular category $\catN$ is non-degenerate.
\end{conjecturei}

\begin{conjecturei}\label{conj:modular}
  The category $\catN$ is a relative modular $\C^*$-category.
\end{conjecturei}

As mentioned above, Conjecture \ref{conj:modular} implies Conjecture
\ref{conj:non-degenerate}.

\begin{theoremi}\label{T:non-degenerateSL21}
  Conjecture \ref{conj:non-degenerate} is true for $\g=\sll(2|1)$. 
\end{theoremi}
The following theorem  together
with \cite{CGP14} implies the existence of 3-dimensional invariants (in text this is Theorem \ref{T:non-degenerateSL21b}):

\begin{corollaryi}[Modified $3$-manifold invariants from perturbative modules]
  There is an renormalized R-T invariant of 3-manifold associated to
  $\catN$ in the case of $\g=\sll(2|1)$.
\end{corollaryi}
This 3-manifold invariant is an extension of the colored Links-Gould
invariant of links (\cite{LG,GP1}) which is analogous to the original
Witten-Reshetikhin-Turaev 3-manifold invariant extending the colored
Jones polynomial \cite{Jo}.  It should be noted that the case of
$\sll(2|1)$ has already been studied: a) Theorem \ref{THEOREM1}, in
the case of $\sll(2|1)$, was proved by Ha, see \cite{Ha,Ha2}, b) in
\cite{AG17}, the category $\catN$, in the case of $\sll(2|1)$, was
shown to be a relative $G$-spherical category and with the machinery
of \cite{GP13} leads to a modified Turaev-Viro invariant.

We find the construction of this paper interesting both from an algebraic point of view and also because of its topological applications.  In particular, the notion of taking a quotient by negligible morphisms of an m-trace in the perturbative case seems to have value and should be studied further.  It is the first mathematical example the authors know of where the modules used in the construction of a pre-relative modular category are deformations of the underlying non-quantized Lie (super)algebra.  However, similar  observations were used in mathematical physics by Mikhaylov and Witten in \cite{MWi}, where they considered super-group Chern-Simons theories.  Further work in this area appeared in \cite{AGPS, Mik2015}.  As in the case of Witten's Quantum Field Theory interpretation of the Jones polynomial, it is interesting to ask if the invariants in this paper are perturbative versions of the invariants discussed in \cite{MWi}?  Even if the answer to this question is no, it would be interesting to investigate if the invariants of this paper have a physical interpretation, e.g.\ a meaning in super-group Chern-Simons theories.  Note that in \cite{MWi}, the question of the existence of invariants of higher rank than $\gl(1|1)$ was not understood and even questioned to be possible.  However, the existence and use of the m-traces makes this possible.

{\bf Acknowledgments.} N.\ Geer was partially supported by the NSF grant DMS-1452093. C.\ Anghel acknowledges the support of the European Research Council (ERC) under the European Union's Horizon 2020 research and innovation programme (grant agreement No 674978).

{}
\section{Preliminaries}
In the section we review background material that will be used in the
following sections.

A \emph{super-space} is a $\Z_{2}$-graded vector space
$V=V_{\p 0}\oplus V_{\p 1}$ over $\C$.  We denote the parity of a
homogeneous element $x\in V$ by $\p x\in \Z_{2}$.  We say $x$ is even
(odd) if $x\in V_{\p 0}$ (resp.  $x\in V_{\p 1}$).  A \emph{Lie
  superalgebra} is a super-space
  with
a super-bracket
that
preserves the $\Z_{2}$-grading, is super-antisymmetric
and satisfies the super-Jacobi
identity (see \cite{K}).  Throughout, all modules will be
$\Z_{2}$-graded modules (module structures which preserve the
$\Z_{2}$-grading, see \cite{K}).

\subsection{The super Lie algebra $\slmn$}\label{SS:SuperLieAlg}
Let $\g=\g_{\p 0}\oplus \g_{\p 1}$ be the special linear super Lie
algebra $\slmn$ with $m \neq n$.  Its rank denoted
$\rk$ is $\rk=m+n-1$. Let $\borel$ be the distinguished Borel
sub-superalgebra of $\g$.  Then $\borel$ can be written as the direct
sum of a Cartan sub-superalgebra $\h$ and a positive nilpotent
sub-superalgebra $\nil_{+}$.  Moreover, $\g$ admits a decomposition
$\g=\nil_{-}\oplus \h \oplus \nil_{+}$.

We can identify $\slmn$ with the super Lie algebra of super-trace zero
$(m|n)\times(m|n)$ matrices.
This standard representation is generated by the elementary matrices
$E_i=E_{i,i+1}$, $F_i=E_{i+1,i}$, $H_m=E_{m,m}+E_{m+1,m+1}$ and
$H_i=E_{i,i}-E_{i+1,i+1}$ for $i=1\cdots \rk$, $i\neq m$. The Cartan subalgebra $\h$
with basis $\{H_i\}$ is contained in the space of diagonal matrices
$X$.  The space $X^{*}$ has a canonical basis
$(ε_1,\ldotsε_{m+n})$ which is dual to the basis formed
by the matrices $E_{i,i}$. Set $δ_{i}=ε_{i+m}$, then $\h$
is the kernel of the super-trace $\str=\sum_{i=1}^mε_i-\sum_{i=1}^nδ_j$.
Therefore, $\h^*$ is the quotient of $X^{*}$ by the super-trace.

A Cartan matrix associated to a Lie superalgebra is a pair consisting
of a matrix $A=(a_{ij})$ and a set determining the parity of the
generators.  For $\g=\slmn$ we choose the matrix determined by $[H_i,E_j]=a_{ij}E_j$.  It is the Cartan matrix associated to
the Dynkin diagram
$$\epsh{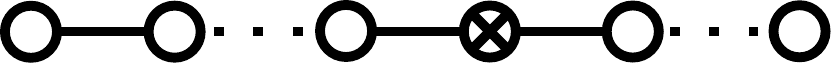}{2ex}\vspace*{1ex}$$
see \cite{K}.  Then this Cartan matrix is the
$\rk\times\rk$-matrix
$A=\left(a_{ij}\right)$ whose non zeros
entries are given by
$$
a_{i,i}=2, \;\; a_{i,i+1}=-1,\;\; a_{i+1,i}=-1, \;\; a_{m+1,m}=-1,
\;\; a_{m,m}=0, \;\; a_{m,m+1}=1
$$
for $i\neq m$ and the set determining the parity is $\{m\}$.  By
setting $d_i=1$ for $i=1,\ldots, m$ and $d_i=-1$ for $i>m$ then
$(d_{i}a_{ij})$ is a symmetric matrix.

Let $\brk{\cdot,\cdot}$ be the symmetric non-degenerate form on $\h$
determined by $\brk{H_{i},H_{j}}=d_{j}^{-1}a_{ij}$ which is equal to
the restriction on $\h$ of $\brk{H,H'}=\str(H.H')$. So it extends to the
whole set of diagonal matrices and induces on its dual the bilinear
form defined by $\brk{ε_i,ε_j}=δ^i_j$ for $i,j=1,\dots, m$ and
$\brk{δ_i,δ_j}=-δ^i_j$ for $i,j=1,\dots, n$.  Hence $\h^*$ can also be
identified as an euclidean space with $\str^\perp$ ($\str$ is not
isotropic since $m\neq n$).  Moreover, the form $\brk{\cdot,\cdot}$ induces a
Weyl-invariant bilinear form on $\h^{*}$, which we will also denote by
$\brk{\cdot,\cdot}$.

Let $\roots^+=\roots_{\p 0}^{+} \cup \roots_{\p 1}^{+}$ to be the set of positive roots where 
$$\roots^+_{\p 0}=\{ε_i-ε_j,\,1\leq i<j\leq m\}\cup
\{δ_i-δ_j,\,1\leq i<j\leq n\}$$
is the set of even positive roots and 
$$\roots^+_{\p 1}=\{ε_i-δ_j\}$$
is the set of odd positive roots.
Let $\roots^-=\roots_{\p 0}^{-} \cup \roots_{\p 1}^{-}$ to be the set of negative roots where 
$$\roots^-_{\p 0}=\{-\alpha \; |\; \alpha\in \roots_{\p 0}^{+}\} \text{ and } \roots^-_{\p 1}=\{-\alpha \; | \; \alpha\in \roots_{\p 1}^{+}\}.$$
A positive root is called
\emph{simple} if it cannot be decomposed into a sum of two positive
roots.  
We follow the normal ordering on $\roots^+$ given in
\cite[Section 3.2]{Yam94}: the positive roots $\ve_i-\ve_j$ are ordered as
couples $(i,j)$ with the lexicographic order.
We denote the \emph{simple roots} as $\alpha_i=\epsilon_i-\epsilon_{i+1}$ for $1\leq i \leq m+n-1=\rk$.
 The values of the form $\brk{\cdot,\cdot}$ on these simple roots are given by 
$$
\brk{\alpha_i,\alpha_j}=d_ia_{i,j}.
$$
Also, note that $\alpha_i(H_j)=a_{j,i}$.  
We denote with $\Lambda_R$ the \emph{root lattice}, which is the $\Z$-lattice generated by the simple roots $\{\alpha_i\}$.

 Let $\rho_{\p0}$ (resp. $\rho_{\p1}$) denote the half sum of
all the even (resp. odd) positive roots and set $\rho= \rho_{\p0}-\rho_{\p1}$.  Then we have
$$2\rho_{\p0}=\sum_i (m+1-2i)ε_i+\sum_j (n+1-2j)δ_j\quad
\text{,}\quad 2\rho_{\p1}=n\sum_i ε_i-m\sum_j δ_j$$
$$\text{and } \rho=\rho_{\p0}-\rho_{\p1}=\frac12\left(\sum_i
  (m-n+1-2i)ε_i+\sum_j (m+n+1-2j)δ_j\right).$$ 
and for any simple root $\alpha_i$,
\begin{equation}
  \label{eq:rho}
  \brk{2\rho,\alpha_i}=\brk{\alpha_i,\alpha_i}.
\end{equation}
Let $w_1,...,w_{\rk}$ be the fundamental dominant weights, which are the vectors of $\h^*$ determined by the condition $ \brk{w_i,\alpha_j}=d_i\delta_{i,j}$ for all $i,j=1,...,\rk$.  
Up to the super-trace, representatives of the fundamental dominant weights are given by
\begin{align*}
  w_k&=\sum_{i=1}^kε_i\text{ for }k=1\cdots m-1\\
  w_m&=\sum_i ε_i=\sum_j δ_j=\frac{-2\rho_1}{m-n}\\
  w_{m+k}&=-\sum_{j=k+1}^{n}δ_j\text{ for }k=1\cdots n-1.\\
\end{align*}
Let $\Lambda_W$ the \emph{weight lattice}, which is the $\Z$-lattice generated by the  fundamental dominant weights.

\subsection{Two super quantum groups associated to $\slmn$} 
We assume that $m\neq n$ are non zero integers, $\ell\ge \rk=m+n-1$
is an odd integer and $ξ$ is the root of unity
$ξ=\exp\bp{2i\pi/\ell}$.
As we will see later, the restriction on the parity of $\ell$
comes from Theorem \ref{T:Andersen} and the choice $\ell\ge {\rk}$ ensures
the semi-simplicity of the tensor square of a typical module (see Lemma \ref{L:21}).
Also, let $h$ be an indeterminate and set $q=\e^{h/2}$.  Next we discuss two different quantum groups associated to $\slmn$.  
  We adopt the following
notations:
$$q^z=\e^{zh/2}\quad\textrm{and}\quad\qn z_q=q^z-q^{-z}$$
for $z\in \C$.
\begin{definition}[\cite{Yam94}]\label{D:Usl}
  Let $U_h(\g)=U_h\slmn$ be the $\C[[h]]$-superalgebra generated by the
  elements $H_{i},E_{i}$ and $ F_{i}, $ $i = 1,\dots, \rk$, 
  satisfying the relations:
  \begin{align*}
    [H_{i},H_{j}] &=0, & [H_{i},E_{j}]=&a_{ij}E_{j}, &
    [H_{i},F_{j}]=&-a_{ij}F_{j},
  \end{align*}
  \begin{align*}
    [E_{i},F_{j}]=&δ_{i,j}\frac{q^{d_iH_{i}}-q^{-d_iH_{i}}}{q^{d_i}-q^{-d_i}}, &
    E_{m}^{2}=&F_{m}^{2}=0,
  \end{align*}
 \begin{align}\label{E:Qserre2}
   [E_i,E_j]=&0 \text{ if } |i-j|>2, 
 \end{align}
 \begin{align}\label{E:Qserre3}
   E^2_iE_j-(q+q^{-1})E_iE_jE_i+E_jE_i^2=
   & 0 \text{ if } |i-j|=1 \text{ and } i\neq m, 
 \end{align}
 \begin{multline}\label{E:Qserre4}
   E_{m}E_{m-1}E_{m}E_{m+1}+E_{m}E_{m+1}E_{m}E_{m-1}+E_{m-1}E_{m}E_{m+1}E_{m}
   \\+E_{m+1}E_{m}E_{m-1}E_{m}-(q+q^{-1})E_{m}E_{m-1}E_{m+1}E_{m}=0.
\end{multline}
and  the relations \eqref{E:Qserre2}-\eqref{E:Qserre4} with $E$ replaced by $F$.  All generators are even
except for $E_m$ and $F_m$ which are odd and $[\cdot,\cdot]$ is the super-commutator given by
  $[x,y]=xy-(-1)^{\p x \p y}yx$.  
\end{definition}

The super algebra
$U_h(\g)$ is a Hopf superalgebra, see Yamane \cite{Yam94}.
 The coproduct, counit and antipode are given by
  \begin{align*}
    \label{}
    \Delta {E_{i}}= & E_{i}\otimes 1+ q^{-d_iH_{i}} \otimes E_{i}, &
    ε(E_{i})= & 0 & S(E_{i})=&-q^{d_iH_{i}}E_{i}\\
    \Delta {F_{i}} = & F_{i}\otimes q^{d_iH_{i}}+ 1 \otimes F_{i}, &
    ε(F_{i})= &0 & S(F_{i})=&-F_{i} q^{-d_iH_{i}}\\
    \Delta {H_{i}} = & H_{i} \otimes 1 + 1\otimes H_{i}, & ε(H_{i}) = & 0 &
    S(H_{i})= &-H_{i}.
  \end{align*}
\begin{remark}
  All the papers \cite{KTol, Yam94, Z2} use different conventions for
  the coalgebra structure of $U_h(\g)$.  Here we use the
  conventions chosen by \cite{KTol}, our coalgebra structure can be
  recovered from \cite{Yam94} by changing $h$ to $-h$ and $q$ to
  $q^{-1}$.  Finally, we point out that there is a mistake in the Hopf
  algebra structure of $U_h(\g)$ in \cite{GP2} and they should be
  as above (in other words, one needs to replace $H_i$ with $d_iH_i$).
  This mistake did not affect the computations of \cite{GP2}.
\end{remark}
Khoroshkin, Tolstoy \cite{KTol} and Yamane \cite{Yam94} showed that
$U_h(\g)$ has an explicit $R$-Matrix ${\Rm}$.  We will now recall
this work, following Yamane.  First, let $(d_{ij})$ be the inverse of
the matrix $(a_{ij}/d_j)$.  Set
\begin{equation}
  \label{E:FormOfK}
  \HR^h=q^{\sum_{i,j}^{\rk}d_{ij}H_{i}\otimes H_{j}}.
\end{equation}
  Then the $R$-matrix is of the form
\begin{equation}
  \label{eq:Rh}
  {\Rm}^h=\check{{\Rm}}^h\HR^h.
\end{equation}
where $\check{{\Rm}}^h$ is described as follows.  

 For $\alpha \in \roots^{+}$, let
 $E_{\alpha}$ and $F_{\alpha}$  be the $q$-analogs
 of the Cartan-Weyl generators where $E_{\alpha_{i}}=E_{i}$ and
 $F_{\alpha_{i}}=F_{i}$ for any simple root $\alpha_{i}$ (see Section 5.1.1 of \cite{Yam94} for the definition of $E_\alpha$ then $F_\alpha$ comes from the dual Hopf algebra: $F_\alpha=E_\alpha^{\circ}\sigma^{p(\alpha)}$ in the notation of \cite{Yam94}).
For $\alpha \in \roots^{+}$,
 let 
$q_\alpha=(-1)^{\p \alpha}q^{-\brk{\alpha,\alpha}}$, where $\p \alpha$ is
the parity of $E_\alpha$.    Lemma 10.3.1 and Theorem 10.6.1 of \cite{Yam94} imply that there exist $c_\alpha=(-1)^iq^k$ for some $i,k\in \Z$ such that 
the element given in Equation \eqref{eq:Rh} is an
$R$-matrix where 
\begin{equation}
  \label{E:Rcheck}
  \check{{\Rm}}^h=\prod_{\alpha \in\roots^{+}}\exp_{q_\alpha}\left(c_{\alpha}(q-q^{-1})(E_{\alpha}\otimes F_{\alpha})\right),
\end{equation}
here the order in the product is given by the chosen fixed normal
ordering of $\roots^{+}$.

In the next definition we think of $q$ as a formal indeterminate.  Let $\K_\ell$ be the subring of $\C(q)$ made of fractions that have no poles at $ξ$ ($\K_\ell$ is a localization of $\C[q]$). A $\K_\ell$-module can be specialized at $q = ξ$ (the specialization is the tensor product with the $\K_\ell$-module $\C$ where $q$ acts by $ξ$).
 
 \begin{definition}\label{relqgp}
Let $U_{q}(\g)=U_{q}\slmn$ be the $\K_\ell$-superalgebra generated
  by the elements $K^{\pm}_i,E_{i}$ and $ F_{i}, $ $i = 1,\dots, \rk$, satisfying the
  relations:
 \begin{align}
 \label{E:1KK}
 K_i  K_j &=  K_j K_i , &  K_iK_i^{-1}&=K_i^{-1} K_i=1 ,
\end{align}
\begin{align}
 \label{E:1KEK}
K_iE_jK_i^{-1}&=q^{d_ia_{ij}}E_j, & K_iF_jK_i^{-1}&=q^{-d_ia_{ij}}F_j, 
 \end{align}
\begin{align}
 \label{E:1EF}
[E_{i},F_{j}]=&δ_{i,j}\frac{K_i-K_i^{-1}}{q^{d_i}-q^{-d_i}}, & E_m^2=F_m^2=0
\end{align}
together with the quantum Serre type relations:
\eqref{E:Qserre2}-\eqref{E:Qserre4} and the $F$ analogs.  All
generators are even except for $E_m$ and $F_m$ which odd.  As above,
here $E_{\alpha}$ and $F_{\alpha}$ for $\alpha \in \roots^{+}$ are the
$q$-analogs of the Cartan-Weyl generators.
\end{definition}

Then, $U_{q}(\g)$ is a Hopf superalgebra with the following coproduct, counit and
antipode:
\begin{align*}
\label{}
\Delta ({E_{i}})= & E_{i}\otimes 1+ K_i^{-1} \otimes E_{i}, & ε(E_{i})= & 0, & S(E_{i})=&-K_iE_{i},\\
\Delta ({F_{i}}) = & F_{i}\otimes K_i+ 1 \otimes F_{i}, & ε(F_{i})= &0, & S(F_{i})=&-F_{i} K_i^{-1},\\
\Delta(K_i^\pm)=&K_i^\pm\otimes K_i^\pm, &ε(K_i^\pm)=&1, & S(K_i^\pm)=&K_i^{\mp}.
 \end{align*}
The \emph{unrestricted quantum group} $U_{ξ}(\g)$ is the $\C$-superalgebra obtained from $U_{q}(\g)$ by specializing $q$ to $ξ$.

\begin{definition}[Unrolled quantum group]
Let $U_{ξ}^{H}(\g)=U_{ξ}^{H}\slmn$ be the $\C$-superalgebra generated
  by the elements $H_{i}, K^{\pm}_i,E_{i}$ and $ F_{i}, $ $i = 1,\dots, \rk$, satisfying 
 Relations \eqref{E:1KK}-\eqref{E:1EF} and the quantum Serre type Relations
\eqref{E:Qserre2}-\eqref{E:Qserre4} and the $F$ analogs where $q = ξ$, plus the relations
 \begin{align*}
 [H_i,H_j] &=[H_i,K^\pm_j]=  0 , &  [H_i, E_j] &=a_{ij}E_j , & [H_i,F_j]&=-a_{ij}F_j,
\end{align*}
 \begin{align*}
 E_{\alpha}^\ell= F_{\alpha}^\ell=&0,  \text{ for all } \alpha \in \Delta^{+}_{\p0}.
\end{align*}
 All generators are even except $E_m$ and $F_m$ which are odd.  As above,
here $E_{\alpha}$ and $F_{\alpha}$ for $\alpha \in \roots^{+}$ are the
$q$-analogs of the Cartan-Weyl generators.
\end{definition}
The algebra $U_{ξ}^{H}(\g)$ is a Hopf superalgebra with coproduct $\Delta$, counit $\epsilon$ and antipode $S$, which is defined as above on $K^{\pm}_i, E_{i}$ and $ F_{i}$ and defined on the elements $H_i$ for $i = 1, . . . , r$ by
\begin{align*}
\Delta ({H_{i}}) = & H_{i} \otimes 1 + 1\otimes H_{i}, & ε(H_{i})  = & 0, & S(H_{i})= &-H_{i}.
 \end{align*}
 \begin{definition}[Category of weight modules $\catn$]\label{D:catn}
  We consider the category $\catn$ of $U^{H}_{ξ}(\g)$
weight modules. Its objects are finite dimensional $U^{H}_{ξ}(\g)$
supermodules $V$ with the following properties:
\begin{enumerate}
\item The element $H_i$ acts diagonally on $V$ for all $i=1,...,{\rk}$.
\item The element $K_i$ acts as $ξ^{d_iH_i}$:
  $$\rho_V(K_i)=\exp\bp{\frac{2\sqrt{-1}\pi}\ell \rho_V(d_iH_i)}\in\End_\C(V), \text{ for all }  i=1,...,{\rk}.$$
\end{enumerate}
The morphisms of $\catn$ are equivariant linear maps that respect the parity.
\end{definition}
If $V$ is an object in $\catn$ then the first condition of the previous definition says that there exists a basis $\{v_j\}$ of $V$ and  $\lambda_j \in \h^*$ such that $H_iv_j=\lambda_j(H_i)v_j$ for all $i,j$.  The linear functional  $\lambda_j \in \h^*$ is called the  \emph{weight} of $v_j$.

If
$\alpha=\sum_i{n_i\alpha_i}\in\Lambda_R$ is an element of the root
lattice, then we define
$K_\alpha=\prod_iK_i^{n_i}\in
U^{H}_{ξ}(\g)$.
Then $K_i=K_{\alpha_i}$ and for a weight vector $v$ of weight $\mu\in\h^*$ in an object of $\catn$, one has
\begin{equation}
  \label{eq:Ki}
  K_\alpha\cdot v=\xi^{\brk{\alpha,\mu}}v.
\end{equation}

\begin{definition}[Perturbative modules]\label{D:perturb}
A weight $\lambda\in \h^*$ is \emph{perturbative} if $\lambda(H_i)\in \Z$ for all $i\neq m$.
An object of $\catn$  is \emph{perturbative}, if all its weights are perturbative.
\end{definition}

Every simple finite dimensional weight $U^{H}_{ξ}(\g)$-module has a unique 
highest weight $\lambda\in\h^*$.
At a root of unity, any weight $\lambda\in\h^*\simeq \C^r$ is the
highest weight of a simple weight $U^{H}_{ξ}(\g)$-module.  This is in contrast
with the h-adic setting where a necessary condition for
$\lambda\in\h^*$ to be the highest weight of a finite dimensional module
is that $\lambda$ is perturbative.  

\begin{definition}[Category of perturbative modules $\catp$]\label{D:PertMod}
Let us consider $\catp$ to be the full subcategory of $\catn$ which has as objects the set of perturbative modules.
\end{definition}

\subsection{Relative (pre-)modular category}\label{SS:RelModDef}
In this section we recall briefly the main theoretical notion used to
define TQFTs as in \cite{CGP14} and \cite{D17}.

Let $\Gr$ be a commutative group.  The notion of a relative
$\Gr$-modular category appeared in \cite{CGP14}, in order to build
invariants of decorated 3-manifolds (where the decoration include a
$\Gr$-valued $1$-cohomology class).  We should call these categories
non-degenerate relative pre-modular as De Renzi gave in \cite{D17} an
additional ``modularity condition'' that ensures the existence of an
underlying $1+1+1$-TQFT.

The construction is based on the notion of modified trace (or m-trace
for short) which replaces the usual categorical trace in modular
categories.  Let $\catc$ be a linear ribbon category over a field
$\kk$ (a ribbon category with unit $\unit$ where the hom-sets are
$\kk$ vector spaces, the composition and tensor product of morphisms
are $\kk$-bilinear, and the canonical $\kk$-algebra map
$\kk \to \End_\catc(\unit), k \mapsto k \, \Id_\unit$ is an
isomorphism).  An object $V\in\catc$ is simple if
$\End_\catc(V)=\kk\Id_V$.  It is regular if its evaluation is an
epimorphism.
\begin{definition}[\cite{GKP1}]\ \\
\noindent a) An {\em ideal} $\ideal$ in 
$\catc$ is a full subcategory which satisfies:
\begin{enumerate}
\item stable by retract: if $W$ in $\ideal$, $V$ is an object of $\catc$ and there exist $\alpha:V\to W$ and $\beta:W\to V$ morphisms such that $\beta\alpha=\Id_{V}$ then $V\in \ideal$.

\item absorbing for the tensor product: if $U\in\ideal$ then for all
  $V \in \catc$, $U\otimes V\in\ideal$.
\end{enumerate}
b) An \emph{m-trace} on an ideal $\ideal$ is a family of linear functions 
$\{\mt_V\}_{V\in\ideal}$ where $\mt_V:\End_\catc(V)\to\kk$ satisfies:
\begin{enumerate}
 \item cyclicity property: $\mt_V(fg)=\mt_U(gf)$, 
\item  partial trace
property: 
$\mt_{U\otimes V}(f)=\mt_U(\ptr_V(f))$, for all $f\in\End(U\otimes V)$
\end{enumerate}
where $\ptr_V(f)$ is the right partial trace of $f$ obtained by using the duality morphism to close $f$ on the right. \\
c) Given an m-trace $\mt$ on $\ideal$, the {\em modified dimension} of $V\in \ideal$ is defined to be $\qd(V)=\mt_V(\Id_V)$.
\end{definition}
Remark that the notion of right and left ideals exist for pivotal
categories but they coincide in a ribbon category.

\begin{definition}
  
  Let $\catb$ be a 
  linear monoidal category. 

$1)$ A set of objects $ \mathcal D=\{ V_i \mid i \in J \} $ of $\catb$ is called a \emph{dominating set} if for any object $V \in \catb$ there exist indices $\{i_1,..,i_m \} \subseteq J $ and morphisms 
$\iota_k \in \Hom_{\catb}(V_{i_k},V), s_k \in \Hom_{\catb}(V,V_{i_k}),$ for all $ k\in \{1,...,m\}$ such that:
$$\Id_{V}=\sum_{k=1}^m \iota_k \circ s_k.$$ 

$2)$ A dominating set $\mathcal D$ is \emph{completely reduced} if:
$$dim_\kk(\Hom_\catb(V_i,V_j))=\delta_{i,j}, \text{ for all } i,j \in J.$$
\end{definition}

\begin{definition}[Free realisation] Let $\Zt$ be a commutative group with additive notation.  A \emph{free realisation} of $\Zt$ in $\catc$ is a set of objects $\{\sigma(k)\}_{k\in \Zt}$ such that 
\begin{enumerate}
  \item $\sigma(0)=\unit$,
  \item the quantum dimension of $\sigma(k)$ is $\pm 1$ for all $k\in \Zt$,
  \item $\sigma(j)\otimes \sigma(k)= \sigma(j+k)$ for all $j,k\in \Zt$,
  \item $ θ_{\sigma(k)}=\Id_{\sigma(k)}, \text{ for all } k \in \Zt $
where $\{\theta_V:V\to V\}_{V\in \catc}$ is the twist in $\catc$,
\item for any simple object $V$ in $\catc$ we have $V\otimes \sigma(k) \cong V$ if and only if $k=0$.  
\end{enumerate}
\end{definition}

\begin{remark}
The above definition was first given in \cite{CGP14} with the condition that the quantum dimensions are all $1$.  This condition was relaxed in \cite{D17}.  Our definition of a free realisation is equivalent to the one given in \cite{D17}.  
\end{remark}

\begin{definition}[A $\Gr$-grading on a category] \label{D:Gstr}
Let $\Gr$ be a commutative group with additive notation.   A \textit{$\Gr$-grading} on $\catc$ is an equivalence of linear categories
$\catc \cong \bigoplus_{g \in \Gr} \catc_g$ where 
$\{ \catc_g \mid g \in \Gr \}$ is a  family of full subcategories of $\catc$
satisfying the following conditions:
  \begin{enumerate}
  \item $\unit \in \catc_0$,
  \item  if $V\in\catc_g$,  then  $V^{*}\in\catc_{-g}$,
  \item  if $V\in\catc_g$, $V'\in\catc_{g'}$ then $V\otimes
    V'\in\catc_{g+g'}$,
  \item  if $V\in\catc_g$, $V'\in\catc_{g'}$ and $\Hom_\catc(V,V')\neq 0$, then
    $g=g'$.  
    \end{enumerate}
\end{definition}

\begin{definition}
Let $\Gr$ be a commutative group and a subset $\XX \subset \Gr$. We define the following two notions:
\begin{enumerate}
\item $\XX$ is \emph{symmetric} if $\XX=-\XX$.
\item $\XX$ is \emph{small} in $\Gr$ if for all $g_1,\ldots ,g_n\in \Gr$ we have:  
  $$ \bigcup_{i=1}^n (g_i+\XX) \neq \Gr.$$
\end{enumerate}

\end{definition}
 \begin{definition}(Relative pre-modular category)\label{D:1}
 Let $\Gr$ and $\Zt$ be commutative groups and $\XX$ be a small symmetric subset of $ \Gr$. Let $\catc$ be a linear ribbon category over a field $\kk$ with the following  data:
 \begin{enumerate}
 \item a $\Gr$-grading on $\catc$,
 \item a free realisation $\{\sigma(k)\}_{k\in \Zt}$ of $\Zt$ in $\catc_0$,
 \item
 a non-zero m-trace $\mt$ on the ideal of projective objects of $\catc$.
 \end{enumerate}
 A category $\catc$ with this data is called a \emph{pre-modular $\Gr
$-category relative to $(\Zt,\XX)$} if the following properties are satisfied:
 \begin{enumerate}
\item \emph{Generic semisimplicity}: \label{D:5} for every $g \in \Gr \setminus \XX$, there exists a finite set of regular simple objects:
$$\Theta(g):=\{ V_i \mid i \in I_g  \}$$
such the following set is a completely reduced dominating set for $\catc_g$:
$$\Theta(g) \otimes \sigma(\Zt):=\{ V_i \otimes \sigma(k) \mid i \in I_g, k\in \Zt  \}.$$

\item \emph{Compatibility}: \label{D:6}There exist a bilinear map $\psi: \Gr \times \Zt \rightarrow \kk^*$ such that:
$$c_{\sigma(k),V}\circ c_{V,\sigma(k)}= \psi(g,k) \cdot  \Id_{V \otimes \sigma(k)}.$$
for any $g\in \Gr$, $V \in \catc_g$ and $k \in \Zt$.
\end{enumerate}
 \end{definition}
In the following, we present the extra requirements that are needed in order to have a relative modular category.
\begin{definition}[Kirby color and Non-degenerate]\label{d:ndeg}
Let $\catc$ be a pre-modular $\Gr$-category relative to $(\Zt,\XX)$.

\begin{enumerate}
  \item For $g \in \Gr \setminus \XX$, let the \emph{Kirby color of index $g$} be the following formal linear combination of objects: $$\Omega_g:= \sum_{i \in I_g}d(V_i) \cdot V_i.$$  \item For $g \in \Gr \setminus \XX$ and $V\in \catc_g$ consider the scalars $\Delta_\pm\in \kk$ defined by
$$\epsh{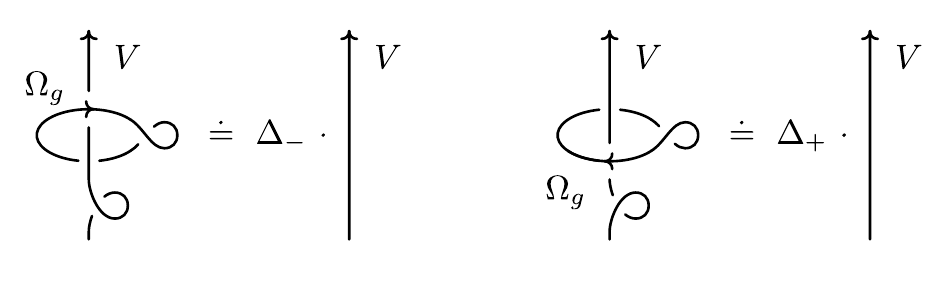}{20ex}
$$
where $ \dot{=}$ means that graphs are \emph{skein equivalent}, i.e.\ equal up to the Reshetikhin-Turaev functor associated with the category $\catc$, see for example Section 1.2 of \cite{D17}.   Lemma 5.10 of \cite{CGP14} implies these scalars do not depend neither on $ V$ nor on $g$. We say $\catc$ is \emph{non-degenerate} if $\Delta_{+}\Delta_{-}\neq 0$. 
\end{enumerate}
\end{definition}
We now state the additional condition required for a modular relative category.  
\begin{definition}[Relative modular category] \label{D:2}

We say that a category $\catc$ is a modular $\Gr
$-category relative to $(\Zt,\XX)$ if:
\begin{enumerate}
\item $\catc$ is a pre-modular $\Gr$-category relative to $(\Zt,\XX)$,
\item there exists a modularity parameter $ζ_{\Omega} \in \C^*$ such that for any $g,h \in \Gr \setminus \XX$ and any $i,j \in I_g$ one has:
  \begin{equation}\label{eq:mod}
    \epsh{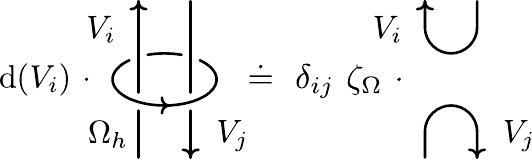}{14ex}\vspace*{4ex}.
  \end{equation}
\end{enumerate}
\end{definition}
As discussed above, non-degenerate pre-modular $\Gr$-categories give rise to invariants of 3-manifolds with some additional structure.
Futhermore modular $\Gr$-categories give rise to TQFTs.  In \cite{D17} it is shown that a modular $\Gr$-category is non-degenerate.

\section{Projective modules at root of unity}\label{S:ProjMod}
\subsection{The typical modules at root of unity}
As above, we assume that $m\neq n$ are non zero integers, $\ell$
is an odd integer greater than ${\rk}=m+n-1$ and $ξ$ is the root of unity
$ξ=\exp\bp{2i\pi/\ell}$.

Recall that a highest weight module is a module generated by a
highest weight vector. The super grading for the algebra $\g=\sll(m|n)$ lifts to a $\Z$-grading leading to a $\Z$-graded Lie algebra (see \cite{Kac78}):
$\g= \g_{-1}\oplus\g_0\oplus\g_1.$
Similarly $\Ux(\frakg)$ is $\Z$-graded:
$$\Ux(\frakg)= \Ux(\frakg_{-1})  \Ux(\frakg_0)  \Ux(\frakg_1).$$
Here $\Ux(\frakg_0)$ is the even part of $\Ux(\frakg)$ and the algebras $\Ux(\frakg_{-1})$ and $\Ux(\frakg_1)$ have a bases $\{F_{\beta_1}^{i_1}F_{\beta_2}^{i_2}...F_{\beta_{mn}}^{i_{mn}}\}$ and $\{E_{\beta_1}^{i_1}E_{\beta_2}^{i_2}...E_{\beta_{mn}}^{i_{mn}}\}$, respectively, where $\{\beta_1,...,\beta_{mn}\}= \roots^+_{\p 1}$ are the odd positive roots and $i_1,...,i_{mn}\in \{0,1\}$.  
\begin{definition}
The {\em typical envelope} of a highest weight $\Ux(\frakg_0)$-module $V_0$
is obtained by first extending the representation $V_0$ over $\Ux(\frakg_0)\Ux(\frakg_1)$, by
the trivial action of $\Ux(\frakg_1)$, then inducing a representation of the super quantum group
$\Ux(\frakg)= \Ux(\frakg_{-1}) \Ux(\frakg_0)  \Ux(\frakg_1)$. 
\end{definition}
\begin{remark}
  Due to the PBW theorem, the character of the typical envelope is the same as the one corresponding to the $\Ux(\frakg_0)$-module $\Ux(\frakg_{-1})\otimes V_0$
(where $\Ux(\frakg_0)$ acts by the adjoint action on $\Ux(\frakg_{-1})$).
  Hence the
typical envelope of a $\Ux(\frakg_0)$ module $V_0$ is a highest weight
$\Ux(\frakg)$-module with the same highest weight and whose dimension is
$2^{mn}\dim(V_0)$.
\end{remark}
Even if $V_0$ is simple as a $\Ux(\frakg_0)$-module, its typical
envelope might not be a simple $\Ux(\frakg)$-module. 
\begin{definition}\label{D:9}  
A highest weight
$\Ux(\frakg)$-module $V$ is called {\em typical} if it is simple and it is a typical envelope. Then its highest weight is also called a {\em typical weight}.  We also say a simple $U_{ξ}^{H}(\g)$-module is {\em typical} if it highest weight is typical.
\end{definition}
To describe the typical modules we need some notation.  The elements of  $\h^*$ (i.e.\ weights) are in one to one correspondence with $\C^{\rk}$:
Given $(c_1,...,c_{\rk})\in \C^{\rk}$ we set $a=c_{m}\in \C$ and $c=(c_1,...,c_{m-1},c_{m+1},...,c_{{\rk}})\in \C^{\rk-1}$ and  write the corresponding weight as \begin{equation}
\label{E:lambdaac}
\lambda_a^c = c_1 w_1+\cdots+c_{m-1}w_{m-1}+aw_m+c_{m+1}w_{m+1}+\cdots+c_{{\rk}}w_{\rk},
\end{equation}
where  $\{w_i\}$ are the 
representatives of the fundamental weights given in Subsection~\ref{SS:SuperLieAlg}.

In \cite{Z2} Zhang, gave a way to compute which $U_h(\g)$-modules are
typical (in the case $\g=\gl(m|n))$.  The next theorem shows a similar computation and gives a characterization of the typical $\Ux(\frakg)$-modules:
\begin{theorem}
  For $c\in \C^{m+n-2}$, $a\in\C$, the weight $\w ac$ is typical if
  and only if
  \begin{equation}
    \label{eq:typ}
    \prod_{\alpha\in\Delta_{\wb1}^+}
  \qn{\brk{\w ac+\rho,\alpha}}_ξ\neq0
  \end{equation}
  where $\qn {z}_ξ=ξ^z-ξ^{-z}$.  
This is equivalent to the following condition for all $ (i,j)\in\{1,\cdots,m\}\times\{1,\cdots,n\}$:
\begin{equation}\label{eq:typ'}
\sum_{k=i}^{m-1}c_k+a-\sum_{k=m}^{m+n-1-j}c_k+(m+1-i-j)\notin
  \frac\ell2\Z  .
\end{equation}
\end{theorem}
\begin{proof}
 This proof is similar to the case for non-quantum classical Lie superalgebra.
Let us consider the elements
$$
\Gamma_-=F_{\beta_1}F_{\beta_2}...F_{\beta_{mn}} \text{ and } \Gamma_+=E_{\beta_1}E_{\beta_2}...E_{\beta_{mn}},
$$
where $\{\beta_1,...,\beta_{mn}\}= \roots^+_{\p 1}$ are the odd positive roots.
   If $V$ is the typical envelope of a module $V_0$
  generated by its highest weight vector $v_0$
   of weight $\w ac$,
 then any non zero
  vector $v$ generates a sub-module which can be shown to contain
  $\Gamma_+\Gamma_-v_0=\lambda\qn1_q^{-mn} v_0$, where $\lambda$ is a scalar.

  Zhang compute the scalar $\lambda$ which reduces to
  $\prod_{\alpha\in\Delta_{\wb1}^+}
  \qn{\brk{\w ac+\rho,\alpha}}_q$
  (up to a unit of the form $\pm q^k$) for the case of a
  perturbative highest weight module.
  But this computation (similar to the Harish-Chandra isomorphism) can
  be performed in $U_q(\g)$ modulo the left ideal generated by
  $\{E_\alpha\}_{\alpha\in\roots_+}$.  Then $\Gamma_+\Gamma_-$ reduces
  to the Laurent polynomial in the $K_i$ given in terms of the
  elements in Equation \eqref{eq:Ki} by
  $\pm q^k{\qn1^{-mn}_q} {\prod_{\alpha\in\Delta_{\wb1}^+}
    \bp{q^{\brk{\rho,\alpha}}K_\alpha-q^{-\brk{\rho,\alpha}}K_\alpha^{-1}}}$).
 Specialized at $q=ξ$
  and evaluated on a highest weight $\w ac\in\C^{{\rk}}$, we obtain
  Equation \eqref{eq:typ}.

  Now if $\Gamma_+\Gamma_-v_0=0$ then $\Gamma_-v_0$ generates a
  strict sub-module of $V$ which is not typical.  Conversely, if
  $\lambda\neq0$ then the sub-module generated by $v$ contains $v_0$
  so it is the whole $V$.  This shows that $V$ is simple.
\end{proof}
\begin{notation}\label{N:Ctypical}
 For any $c\in\C^{m+n-2}$, consider the set 
 $$\TT_c:=\{ a \in \C \mid \w a c \text{ is typical}\}.$$
\end{notation}
\begin{corollary}
  Using the previous characterisation, we conclude that for
  any $c\in\C^{m+n-2}$ the set $\TT_c$ is a dense open set in $\C$.
\end{corollary}

\subsection{$\catn$ is a braided category}
In this subsection we will show that the category $\catn$ is a
braided category. The existence of a braiding will be very useful in the further computation of the m-trace as we will see.
Further on, we will construct a ribbon category from $\catn$. 

For the case of unrolled quantum groups at roots of unity, it was shown in 
\cite{GP13} that their representation theory leads to a ribbon category and in particular it provides a braiding. We aim to obtain the same result, in the context of super quantum groups. We do this in two steps: first we start with the $R$-matrix of the super quantum groups at generic $q$, and truncate it such that it can be evaluated at a particular root of unity. Morally, the issue with the whole generic $R$-matrix is the fact that the denominators vanish for roots of unity, and the idea would be to keep the part of $R$ which gives a well defined specialisation. Then, in step two, we will show that the truncated $R$-matrix leads to an operator which satisfies the braiding properties. Let us make this precise. \\
 
{\bf Step 1-Truncation.}  Yamane showed (\cite{Yam94}) that the $R$-matrix ${\Rm}^h $ of Equation
\eqref{eq:Rh} makes $U_h(\g)$ into a quasi triangular super-Hopf algebra.  In
particular, this means that the following relations are satisfied:
\begin{equation}\label{E:PropertiesRmh}
{(\Delta\otimes\Id ){\Rm}^h={\Rm}^h_{12} {\Rm}^h_{23}}, \;\;\;
{(\Id \otimes \Delta){\Rm}^h={\Rm}^h_{13} {\Rm}^h_{12}}, \;\;\;
{{\Rm}^h\Delta(x)=\Delta^{op}(x) {\Rm}^h,}
\end{equation}
for all  $x \in U_h(\g)$.
Our aim is to construct an $R$-matrix for the unrolled quantum group,
at roots of unity.

\begin{proposition} The algebra
  $U_h(\g)$ has a topological Poincaré-Birkhoff-Witt basis $\calB$ given by the monomials  $$\prod_{i=1}^{{\rk}}H_i^{k_i} \prod_{\beta_i\in
    \roots^+}E_{\beta_i}^{x_i}\prod_{\beta_i\in
    \roots^+}F_{\beta_i}^{y_i},$$ where $x_i,y_i,k_i\in\N$ with 
  $x_i,y_i\le1$ if $\beta_i\in\roots^+_{\p1}$.
\end{proposition}
Let $\calB^<$ be the subset of monomials for which all the $x_i,y_i$
are strictly smaller than $\ell$.  Consider the splitting
$U_h(\g)=U^{<}\oplus I$ where $U^{<}$ is the closure of
$\operatorname{Span}_{\C[[h]]}(\calB^<)$ and $I$ is the closure of
$\operatorname{Span}_{\C[[h]]}(\calB\setminus\calB^<)$.  Let
  $$p: U_h(\g)\rightarrow U^{<}$$
  be the projection map parallel to $I$.
  We will use the tensor product of this projection map, in order to truncate the $h$-adic $R$-matrix.
  Actually it is enough to project just one of its two components:
\begin{definition}[Truncated $h$-adic R-matrix]
We define
\begin{equation}\label{E:DefRm<}
{\Rm}^{<h}=(p \otimes p) \ {\Rm}^{h}=(p \otimes\Id ) \ {\Rm}^{h}=(\Id \otimes p) \ {\Rm}^{h}.
\end{equation}
\end{definition}
\begin{lemma} The truncated $h$-adic $R$-matrix satisfies the following identities:
\begin{enumerate}\label{L:Rmatrix}
\item{ $(p \otimes p \otimes p)(\Delta \otimes\Id )  {\Rm}^{<h}=(p \otimes p \otimes p) ({\Rm}_{13}^{<h} {\Rm}_{23}^{<h})$ }
\item{ $(p \otimes p \otimes p)(\Id \otimes \Delta)  {\Rm}^{<h}=(p \otimes p \otimes p) ( {\Rm}_{13}^{<h} {\Rm}_{12}^{<h})$ }
\item{ $(p \otimes p) {\Rm}^{<h}(\Delta^{op}(x))=(p \otimes p) (\Delta(x)){\Rm}^{<h}, \ \forall x \in U_h(\g)$. }
\end{enumerate}
\end{lemma}
\begin{proof}
The proof is word for word the same as the proof of Proposition 39 in \cite{GP13}.  The main idea is to project the equalities in Equation \eqref{E:PropertiesRmh}. For example, using $p=p \circ p$ and Equation \eqref{E:DefRm<} we have 
$$(p \otimes p \otimes p)(\Delta\otimes\Id ){\Rm}^h=(p \otimes p \otimes p)(\Delta\otimes\Id )(\Id\otimes p){\Rm}^h=(p \otimes p \otimes p)(\Delta\otimes\Id ) {\Rm}^{<h}.$$
Using Equation \eqref{E:PropertiesRmh}, a similar computation shows this is equal to $(p \otimes p \otimes p) ({\Rm}_{13}^{<h} {\Rm}_{23}^{<h})$.  
\end{proof}

{\bf Step 2-Operators at roots of unity.} 
Now, we introduce the truncated matrix, defined in the similar way as the h-adic one, but which uses truncated exponentials instead. The advantage is that this expression will be well defined at roots of unity. Let us make this precise. 

We start with the analog of the Cartan part of the $R$-matrix, which
we will use for the representations over the super quantum group at
roots of unity $U_{ξ}^H(\g)$.
\begin{definition}
Consider the operator in $\catn$ defined as:
$$\HR^{ξ}:=
\exp\bp{\frac{2i\pi}\ell \sum_{i,j}d_{ij}\ H_i \otimes H_j}.$$
\end{definition}

\begin{definition}[Truncated R-matrix]
  Consider the truncated quantum exponential:
  $$\exp_{ξ}^<(x):=\sum_{n=0}^{\ell-1} \frac{x^n}{(n)_{ξ}!}.$$  Similar to Equation \eqref{E:Rcheck}, let
\begin{equation*}
  \check{{\Rm}}^{ξ}=\prod_{\alpha \in\roots^{+}}\exp_{ξ_{\alpha}}^<\big((-1)^{\p
    \alpha}a_{\alpha}^{-1}(ξ-ξ^{-1})(E_{\alpha}\otimes F_{\alpha})\big).
\end{equation*}
Now, we define the expression:
$${\Rm}^{ξ}:=\check{{\Rm}}^{ξ}\HR^{ξ} .$$
\end{definition}

\begin{remark}\label{R:H}
The operators $\HR^h$ on $Rep(U_{h}(\g))$ and $\HR^{ξ}$ on $\catn$ satisfy the same relations. In particular, this leads to the following properties for $\HR^{ξ}$.
\end{remark}

\begin{proposition}\label{P:HH}
  The following operators on objects of $\catn$ coincide:
\begin{enumerate}
\item {
    $\HR^{ξ}(x \otimes y)={ξ}^{\brk{\alpha, \beta}}\big( x
    K_{\beta} \otimes y K_{\alpha}\big) \HR^{ξ}$} where
  $\alpha,\beta$ are the weights of $x,y$ respectively.
\item { $(\Delta \otimes\Id ) \HR^{ξ}=\HR^{ξ}_{13}\HR^{ξ}_{23}.$}
\item{ $(\Id  \otimes \Delta) \HR^{ξ}=\HR^{ξ}_{13}\HR^{ξ}_{12}.$}

\end{enumerate}
\end{proposition}
\begin{proof}
  First remark that in last two equations above the operator $\HR^ξ$
  is defined as a limit of elements of $U^H_{ξ}(\g)^{\otimes2}$ and
  its coproducts should be understood as the limit of the
  coproducts. Now the proposition follows because the Cartan generators
  at roots of unity satisfy the following properties:
 $$ 
   { H_ix=x\bp{H_i+ \frac{1}{d_i}\brk{\alpha_i,w(x)}}},
  \;\;\;\;\; { K_{\alpha_i}={ξ_i}^{H_i}},$$
 \begin{equation}\label{E:ActionHR}{ \HR^{ξ}(v \otimes v')=ξ^{\brk{w(v),w(v')}}v \otimes v'.}
 \end{equation}
  Here, $w(v)$ is the weight of the weight vector $v$.
\end{proof}
Property (1) of Proposition \ref{P:HH} implies that conjugation with 
$\HR^{ξ}$ induces a well defined automorphism of
$U^H_{ξ}(g)^{\otimes2}$.

Summarizing (see Table \ref{Table:SummaryRmatrix}), we have $3$ main steps:
\begin{enumerate}
  \item construct a truncation of the $R$-operator : $ {\Rm}^{ξ}$,
  \item define a Cartan part operator at roots of unity $\HR^{ξ}$,
  \item next, combine them into the braiding ${\Rm}^{ξ}$.
\end{enumerate}
\begin{table}[htp]
\caption{Summary of $R$-martix and operators}
\begin{center}
\begin{tabu} to 1\textwidth { | X[c] | X[c] | X[c] | X[c] |  } 
 \hline
 & $h$-adic super quantum group & truncation & Super quantum group at roots of unity \\ 
\hline                                    
Super quantum group & $U_h(\g)$& projection on $U^< $ & $U^H_{ξ}(\g)$   \\
\hline
Ring& $\C[[h]]$ & $\C(q)$ & $\C$ \\
\hline
quasi R-matrix& $\check{{\Rm}}^h$ & $
\check{{\Rm}}^{<h}=(p \otimes p) \check{{\Rm}}^h$ & $\check{{\Rm}}^{ξ}=\check{{\Rm}}^{<h}\mid_{q=ξ}$ \\
\hline 
Cartan part & $\HR^h$& & $\HR^{ξ}$  \\
\hline
R-matrix & $\Rm^h= \check{{\Rm}}^h\cdot \HR^h $ & ${\Rm}^{<h}=(p \otimes p) {\Rm}^h$ & $ {\Rm}^{ξ}=\check{{\Rm}}^{ξ}\cdot  \HR^{ξ} $  \\
\hline
\end{tabu}
\end{center}
\label{Table:SummaryRmatrix}
\end{table}

\begin{definition}[Projection of the $h$-adic $R$-operator]
In order to have a correspondent of the $R$-operator over $\C[[h]]$, we introduce the following element: $$\check{{\Rm}}^{<h}:=(p\otimes p)\check{{\Rm}}^h.$$
\end{definition}

\begin{lemma}
\label{L:RmOperator}
The following operators
are equal
on objects of $\catn$:
\begin{enumerate}
\item {$(\Delta\otimes\Id ){\Rm}^{ξ}={\Rm}^{ξ}_{13} {\Rm}^{ξ}_{23}$},
\item{$(\Id \otimes \Delta){\Rm}^{ξ}={\Rm}^{ξ}_{13} {\Rm}^{ξ}_{12}$},
\item{${\Rm}^{ξ}\Delta(x)=\Delta^{op}(x) {\Rm}^{ξ},  \text{ for all } x \in U^{H}_{ξ}(\g) $}.
\end{enumerate}
\end{lemma}
\begin{proof}
Let us summarize the strategy to prove this lemma: we use the fact that the projections of ${\Rm}^{<h}$ satisfy the properties from Lemma \ref{L:Rmatrix}. Then, we investigate the implications on the Cartan part and the element $\check{{\Rm}}^h$. Afterwards, we specialise all the relations at roots of unity, and obtain information about the corresponding $\HR^{ξ}$ and $\check{{\Rm}}^{ξ}$. 

We will show the first relation, and the others will follow in an analogous manner.

From Lemma \ref{L:Rmatrix},
we have:
$$(p \otimes p\otimes p)(\Delta \otimes\Id ) (\check{{\Rm}}^{<h}\HR^{h})=(p \otimes p\otimes p) \bp{\check{{\Rm}}^{<h}_{13}\HR^{h}_{13} \cdot \check{{\Rm}}^{<h}_{23}\HR^{h}_{23} }\ $$
The co-product properties for $\HR^h$ (see Remark \ref{R:H} and Proposition \ref{P:HH}) give that:
$$(\Delta \otimes\Id )\HR^h=\HR^{h}_{13}\HR^{h}_{23}.$$
Next we use the facts that $\Delta $ is an algebra morphism and $\HR^{h}_{13} $ is invertible to get
\begin{equation*}\label{eq:0}
(p \otimes p\otimes p)(\Delta \otimes\Id ) (\check{{\Rm}}^{<h})\HR^{h}_{13}\HR^{h}_{23}=(p \otimes p\otimes p) \bp{\check{{\Rm}}^{<h}_{13} \cdot \HR^{h}_{13} \check{{\Rm}}^{<h}_{23} ({\HR^{h}_{13})}^{-1}}\HR^{h}_{13} \HR^{h}_{23}. \ 
\end{equation*}
Now canceling $\HR^{h}_{13} \HR^{h}_{23}$ we have:
\begin{equation}\label{eq:aaa1}
(p \otimes p\otimes p)(\Delta \otimes\Id ) (\check{{\Rm}}^{<h})=(p \otimes p\otimes p) \bp{\check{{\Rm}}^{<h}_{13} \cdot \HR^{h}_{13} \check{{\Rm}}^{<h}_{23} {(\HR^{h}_{13})}^{-1}}. \ 
\end{equation}

The key point of introducing the $h$-adic element $\check{{\Rm}}^{<h}$ is the fact that this has no poles when evaluating $q$ to the root of unity $ξ$, in other words:
$$\check{{\Rm}}^{<h}\mid_{q=ξ}=\check{{\Rm}}^{ξ},$$
and also Proposition \ref{P:HH} Part (1) implies that the conjugation by the the Cartan part commute with specialization at roots of unity:
$$\HR^{h}_{13}\check{{\Rm}}^{<h}_{23} {\HR^{h}_{13}}^{-1}\mid_{q=ξ}=\HR^{ξ}_{13}\check{{\Rm}}^{<ξ}_{23}{\HR^{ξ}_{13}}^{-1}$$
as operators on objects of $\catn$.
This shows, that Equation \eqref{eq:aaa1} can be specialised at roots
of unity.  As the kernel of $p$ specializes to $0$ in $U^{H}_{ξ}(\g)$, it leads to
the following
equality of operators on objects of $\catn$:
\begin{equation}\label{eq:3}
\ (\Delta \otimes\Id ) (\check{{\Rm}}^{ξ}) = \check{{\Rm}}^{ξ}_{13}\cdot \HR^{ξ}_{13}\check{{\Rm}}^{<ξ}_{23}{\HR^{ξ}_{13}}^{-1} .
\end{equation}
Using Proposition \ref{P:HH} Part (2) we can multiply by
$(\Delta \otimes\Id )\HR^{ξ}=\HR^{ξ}_{13}\HR^{ξ}_{23}$, on the
right, and conclude the equality of operators on objects of $\catn$:
$$(\Delta\otimes\Id ){\Rm}^{ξ}={\Rm}^{ξ}_{13} {\Rm}^{ξ}_{23}.$$
\end{proof}

The properties of Lemma \ref{L:RmOperator} directly imply the following corollary.  
\begin{corollary}[Braided category]\label{C:BraidedCat}
 For $V,W \in \catn$ define $c_{V,W} : V \otimes W \rightarrow V \otimes W$ as:
$$c_{V,W}=\tau \circ {\Rm}^{ξ}$$
where $\tau (x \otimes y)= (-1)^{\bar{x}\bar{y}} y \otimes x$ is the super-flip.  Then the family $\{c_{V,W}\}_{V,W\in \catn}$ is a braiding  for $\catn$.
\end{corollary}

\subsection{Pivotal structure}
\begin{notation}
  Let $$\pi=2(1-\ell)\rho_{\p0}-2\rho_{\p1}=2\rho-2\ell\rho_{\p0}\in\Lambda_R,$$
  and the associated element $K_\pi$ (see Equation \eqref{eq:Ki}).
\end{notation}
\begin{lemma}
  The group-like element $K_\pi$ is a pivot for the
  super Hopf algebra $U^{H}_{ξ}(\g)$, meaning that:
$$S^2(u)=K_\pi uK_\pi^{-1}, \forall u\in U^{H}_{ξ}(\g) .$$
\end{lemma}
\begin{proof}
  This can be checked directly on generators using the fact that
  $K_{2\ell\rho_{\p0}}$ is central and Equation \eqref{eq:rho} which
  implies for $i=1\cdots \rk$
$$K_\pi E_iK_\pi^{-1} =K_{2\rho}E_iK_{2\rho}^{-1} =\xi^{\brk{2\rho,\alpha_i}}E_i
=\xi^{\brk{\alpha_i,\alpha_i}}E_i =K_{\alpha_i} E_iK_{\alpha_i}^{-1}
=S^{2}(E_i).$$ A similar computation can be made for $F_i$.
\end{proof}
The pivotal element can be used to define duality morphisms in
$\catn$.
\begin{definition}[Duality]
Given $V\in\catn$, let $V^*$ be the dual super vector space
$\Hom_\C(V,\C)$ endowed with the action
$u\vp=(-1)^{\p u\p \vp}\vp(S(u)\,\cdot\,)$.
Let us consider the following morphisms:
\begin{align*}
  & \ev_V \colon V^*\otimes V \to\unit, \quad \coev_V\colon \unit \to
  V \otimes V^*,
  \\
  & \tev_V \colon V\otimes V^* \to\unit, \quad \tcoev_V\colon \unit
  \to V^* \otimes V,
\end{align*}
defined using a basis $\{v_i\}_i$ of $V$ and the corresponding dual basis  $\{v_i^{*}\}_i$ of $V^*$ by:
\begin{align*}
& \ev_V:v_i^{*}\otimes v_j\mapsto v_i^{*}(v_j)=δ_i^j,  \quad
\coev_V\colon 1\mapsto \sum_i v_i\otimes v_i^{*},\\
&   \tev_V \colon v_j\otimes v_i^{*}\mapsto (-1)^{\p v_i}v_i^{*}(K_\pi.v_j), \quad
\tcoev_V\colon 1\mapsto \sum_i (-1)^{\p v_i}v_i^{*}\otimes(K_\pi^{-1}.v_i).
\end{align*}
\end{definition}
The following result comes from Proposition 2.9 of \cite{B}:
\begin{proposition}
The category $\catn$ together with the morphisms $\ev_V, \coev_V,$ $\tev_V$ and $\tcoev_V$  is a pivotal category.
\end{proposition}

\subsection{$\Gr$-structure on $\catn$}
Recall that $\Lambda_R$ is the root lattice, which is the $\Z$-lattice generated by the simple roots.  Let us
  consider the group $\Gr=\bp{\h^*/\Lambda_R,+}$.
\newcommand{\card}{\operatorname{card}} 
\begin{definition}[Grading on the category $\catn$]\label{D:3}
  For
   $[\lambda]\in \Gr$, let $\lambda\in \h^*$ be any representative of $[\lambda]$.  Define $\catn_{[\lambda]}$ to be the full subcategory
  of $\catn$ with objects consisting of supermodules whose
  weights are all of the form $\lambda+\alpha$ for some $\alpha \in\Lambda_R$.
\end{definition}

Let 
$$\Lambda_\Zt=\qn{\lambda\in\h^*:\,2\brk{\lambda,\alpha_i}\in\ell\Z \text{ for all } i=1,\cdots,\rk}$$
and let $\Lambda_\Zt^0=\Lambda_\Zt\cap\Lambda_R$.  Consider the commutative group 
$$\Zt=(\Z/2\times\Lambda_\Zt^0,+)\simeq (\Z/2\times\Z^{\rk},+).$$

  For $z=(\bar z,\lambda)\in \Zt$, we define the one dimensional $\UH$-module $\sigma{(\bar z,\lambda)}=\C$  as a
vector superspace with super grading $\bar z$ for any vector  and  $\UH$ action given by:
\begin{equation}
  E_i \curvearrowright \sigma{(\bar z,\lambda)}=F_i \curvearrowright \sigma{(\bar z,\lambda)}=0,\quad
  H_i \curvearrowright \sigma{(\bar z,\lambda)} = λ(H_i) \cdot \Id_{\sigma{(\bar z,\lambda)}}.
\end{equation}

\begin{notation}[Critical set $\XX$]\label{N:1} Let $D$ be the
  dimension of the strictly negative Borel:
  $D=\ell^{\card(\roots_{\p0}^-)}2^{\card(\roots_{\p1}^-)}=\ell^{\frac12(m^2+n^2-m-n)}2^{mn}$.
  Since a simple highest weight module $V$ is generated by the action
  of the 
  strictly negative Borel on its highest weight, it must have dimension
  $\dim_\C V\le D$.
  Let us denote by $\XX$ the set of all $x\in \Gr$
  such that there exists a simple highest weight module $V$ in
  $\catn_x$ of dimension $\dim_\C V<D$.
\end{notation}
\begin{theorem}\label{T:SScatn}
  The subset $\XX$ of $\Gr$ is symmetric, small and satisfies the following property:  
for any $[\lambda]\in \Gr\setminus \XX$, the category $\catn_{[\lambda]}$ is
semisimple and has a completely reduced dominating set 
  $$Θ_{[\lambda]} ⊗ \sigma(\Zt)=\{ V_i \otimes \sigma(z) \mid i \in I_{[\lambda]}, z\in \Zt  \}$$ where $Θ_{[\lambda]}$ is a finite set of simple
  modules.
\end{theorem}
\begin{proof}
  First, a module and its dual have the same dimension as complex vector
  space, so $\XX$ is symmetric.  Next, we show $\catn_{[\lambda]}$ satisfies the generically semisimple condition for every $[\lambda]\in \Gr\setminus \XX$.  In this proof, we identify a weight $\mu\in \h^*$ with a tuple $\mu=(\mu_1,...,\mu_\rk)\in \C^\rk$ coming from $\mu=\sum_i\mu_i\alpha_i$.  
  Let $U_{ξ}^{\Lambda_R}(\g)$ be the sub Hopf algebra of $U^{H}_{ξ}(\g)$
  generated by the
  set $\{K_{i},E_i,F_i:i=1\cdots {\rk}\}$ where the $K_i$ generate the
  group ring $\C[\Lambda_R]$.  Here we consider a slightly larger quantum
  group: let
  $U_{ξ}^{\Lambda_W}(\g)=\C[\Lambda_W]\otimes_{\C[\Lambda_R]}U_{ξ}^{\Lambda_R}(\g)$
  be the algebra generated by the set
  $\{K_{w_i},E_i,F_i:i=1\cdots {\rk}\}$ with the usual relations
  including $K_{w_i}E_jK_{w_i}^{-1}=\xi^{\brk{w_i,\alpha_j}}E_j$.
   Let $U_{ξ}^{[\lambda]}(\g)$ be the quotient of
  $ U_{ξ}^{\Lambda_W}(\g)$ by the ideal generated by
  $\{K_{w_i}^\ell-ξ^{\ell\lambda_i}:i=1\cdots {\rk}\}$.  

  Then, PBW theorem implies
  that
  $$
  \dim_\C U_{ξ}^{[\lambda]}(\g) =
  \ell^{{\rk}}\ell^{\card(\roots_{\p0})}2^{\card(\roots_{\p1})} =
  \ell^{m^2+n^2-1}2^{2mn}.
  $$

  On the other hand, forgetting the action of the generator $H_i$, we
  see that every module in $\catn_{[\lambda]}$ has a structure of an
  $U_{ξ}^{[\lambda]}(\g)$-module where $K_{w_i}$ acts as the scalar
  $\xi^{\brk{w_i,\cdot}}$ on weight spaces.  If two simple highest
  weight modules $V,V'$ of $\catn_{[\lambda]}$ are isomorphic as
  $U_{ξ}^{[\lambda]}(\g)$-modules, then there exists $z\in\Zt$ such
  that $V'\simeq V\otimes \sigma(z)$.  Reciprocally, the isomorphism
  class of a simple  $U_{ξ}^{[\lambda]}(\g)$ highest weight 
  module in $\catn_{[\lambda]}$ only depends of its
  highest
  weight $\mu\in\C^{\rk}$ modulo $(\ell\Z)^{\rk}$.
  
  To be in $\catn_{[\lambda]}$, $\mu$ should satisfies
  $ξ^{\ell\mu_i}=ξ^{\ell\lambda_i}$ so there are $\ell^{\rk}$ such modules
  $\{V_i\}$.  Now for generic $c$, the simple $U_{ξ}(\g_{\p0})$-module
  with highest weight $\lambda^a_c$ has dimension
  $\ell^{\card(\roots_{\p0}^-)}$ (see \cite{DK,GP13},
  where generic $c$ means $c$ belongs to some open
  dense set which contains weights $\lambda^a_c$ with $c_i\notin\Z$)
  and its typical envelope (which is simple if $a\in\TT_c$)
  has dimension
  $D=\ell^{\card(\roots_{\p0}^-)}2^{\card(\roots_{\p1}^-)}$.

  The dense open set
  $\{\lambda^a_c:c_i\not\in\Z,a\in\TT_c\}\subset\h^*$ only contains
  highest weight of typical modules of dimension $D$.  Its complement
  is formed by weights whose degree belongs to a small subset of $\Gr$
  containing $\XX$.  Thus $\XX$ is small.
  
  Then the density lemma ensures that the algebra map from
  $U_{ξ}^{[\lambda]}(\g)$ to the $\ell^{\rk}$-fold cartesian product of
  $D\times D$-matrices $\prod_i\End_\C(V_i)$ is surjective.  Counting
  dimensions, we see that this map is an isomorphism thus
  $U_{ξ}^{[\lambda]}(\g)$ is a semi-simple algebra for $[\lambda]\notin \XX$.

  Now for $[\lambda]\notin \XX$, any module $V$ of $\catn_{[\lambda]}$ is semisimple as a
  $U_{ξ}^{[\lambda]}(\g)$-module ; in particular, it has dimension $dD$ for some
  $d\in\N$ and the space of its highest weight vectors
  $W=\bigcap_{i=1}^{\rk}\ker \rho_V(E_i)$ has dimension $d$.  Now $W$ is
  $\Z/2$-graded because each kernel is
  ($\ker \rho_V(E_i)=\ker \rho_V(E_i)_{|V_{\p0}}\oplus\ker
  \rho_V(E_i)_{|V_{\p1}}$) and $W$ is stable for the action of the
  $H_i$. This implies that $W$ has a homogeneous basis $(v_j)_j$ formed
  by weight vectors.  For such a vector, $\UH.v_j=U_{ξ}^{[\lambda]}(\g).v_j$ is
  a simple module of $\catn_{[\lambda]}$ and the map
  $$\bigoplus_j\UH.v_j\to V$$
  is an isomorphism of $\UH$-modules because it is an
  isomorphism of $U_{ξ}^{[\lambda]}(\g)$-modules.  Hence $V$ is semi-simple.
\end{proof}

\begin{remark}
  For generic $\lambda\in\h^*$ as in the above proof, the simple
  module with highest weight $\lambda$ has dimension $D$ and lowest
  weight $\lambda-(\ell-1)(2\rho_{\p0})-2\rho_{\p1}=\lambda+\pi$.  If such a
  module was self-dual, we would have $\lambda+\pi=-\lambda$ thus
  $\lambda=-\pi/2$. But then $\rho-\pi/2=\ell\rho_{\p0}$ and for any odd
  root $\alpha\in\Delta_{\wb1}^+$, $\qn{\brk{-\pi/2+\rho,\alpha}}_ξ=0$
  contradicting Equation \eqref{eq:typ}.  Hence there is no analog of
  the Steinberg module (a simple projective self-dual module) in
  $\catn$.
\end{remark}
\subsection{Ribbon Structure}

\begin{definition}
  Let us consider $θ_V$ to be the right partial trace of the braiding $c_{V,V}$, given by the formula:
$$  
  θ_V=(\Id\otimes\tev_V)(c_{V,V}\otimes\Id)(\Id\otimes\coev_V).
$$  
 The category  $\catn$ is \emph{ribbon} and the family $θ$ is a \emph{twist} if 
 $$
 θ_{(V^*)}=(θ_V)^*, \forall  \ V \in \catn.
 $$
\end{definition}
Next, we show  that $θ$ is a twist.  This will be done first for the case of generic simple modules and it will follow afterwards for all modules, as a consequence of the generic semi-simplicity.
\begin{proposition}\label{P:twist-simple} Let $V$ be a highest
  weight module with highest weight $\lambda$. Then we have the following property:
  \begin{equation}\label{E:CompTwist}
  θ_V=\xi^{\brk{\lambda+\pi,\lambda}}\Id_{V}.
  \end{equation}
   In particular, if $\lambda$
  is generic then $(θ_V)^*=θ_{(V^*)}$.
\end{proposition}

\begin{proof}
By definition of highest and lowest weight vectors, it follows that when computing the braiding of a highest weight vector tensor any vector (or any vector tensor a lowest weight vector) the only part that contributes is $\HR^{ξ}$.  Previously, this observation was used to compute the open Hopf link in Proposition 2.2 of \cite{GP2} and Proposition 45 of \cite{GP13}.  Thus, a direct computation gives Equation \eqref{E:CompTwist}.

  For generic modules, the lowest weight vector is the image of the
  highest weight vector by the lowest weight element of $\UH$,
  i.e. the product of the generators
  $\prod_{\alpha\in\roots_{\p1}^+}F_\alpha\prod_{\beta\in\roots_{\p0}^+}F_\beta^{\ell-1}$.
  Hence the lowest weight vector is equal to the highest weight vector
  plus $\pi=-2(\ell-1)\rho_0-2\rho_1$
 and $V^*$ is a generic simple
  module with highest weight $-\pi-\lambda$.  Then we have
$$θ_{(V^*)}=\xi^{\brk{-\pi-\lambda+\pi,-\pi-\lambda}}\Id_{(V^*)}
  =\xi^{\brk{\lambda,\pi+\lambda}}\Id_{(V^*)}=(θ_V)^*.$$
\end{proof}

\begin{corollary}\label{C:D-ribbon}
  The category $\catn$ together with the twist $θ$ is a ribbon category.
\end{corollary}
\begin{proof}
  By Proposition \ref{P:twist-simple} we have that
  $(θ_V)^*=θ_{(V^*)}$
  for generic simple
  module of $\catn$ thus by \cite[Theorem 9]{GP18}, $θ$ is a twist on
  $\catn$.
\end{proof}

  \subsection{Translation group and free realisation}
Recall the weight lattice $\Lambda_W$, generated by fundamental dominant weights $w_i$ which are determined by $ \brk{w_i,\alpha_j}=d_i\delta_{i,j}$.  Since $d_i=\pm1$, for all $i$,  are invertible in $\Z$ then $\Lambda_W=\qn{\lambda\in\h^*:\,\brk{\lambda,\Lambda_R}\subset \Z}$.   
\begin{lemma} \label{L:LambdaZ0} We have  
  \begin{enumerate}
  \item $\Lambda_\Zt^0=(\ell\Lambda_W)\cap\Lambda_R$,
  \item $\theta_{\sigma{(\bar z,\lambda)}}=\Id_{\sigma{(\bar z,\lambda)}}$, for all $ λ\in\Lambda_\Zt^0$.
  \end{enumerate}
\end{lemma} 
\begin{proof}
 If $\lambda \in \Lambda_R$ then $\lambda=\sum_j n_j\alpha_j$ for $n_j\in \Z$.  It follows that $\brk{\lambda,\alpha_i}\in \Z$ for all $i$.  Thus, since $\ell$ is odd then $\brk{\lambda,\alpha_i}\in \ell\Z$ if and only if $2\brk{\lambda,\alpha_i}\in \ell\Z$.   From our characterization of $\Lambda_W$, presented above, we have: 
 \begin{equation}\label{E:lLambdaW}
 \ell \Lambda_W=\{\lambda \in \h^*:\,\brk{\lambda,\Lambda_R}\subset \ell\Z\}.
 \end{equation}
 Then, the first part of this lemma follows from the last two sentences.
 
 Let $z=(\bar z,\lambda)\in \Zt$ and let $v$ be a generating vector of $\sigma(z)$. Since the root vectors acts trivially on $\sigma(z)$, the only part from the action
of ${\Rm}^{ξ}=\check{{\Rm}}^{ξ}\HR^{ξ}$ contributing to the braiding $c_{\sigma(z),\sigma(z)}$ is $\HR^{ξ}$.  Combining this with the definition of the duality and Equation \eqref{E:ActionHR} we have  
$$\theta_{\sigma{(\bar z,\lambda)}}(v)=ξ^{\brk{\lambda,\lambda}}K_\pi(v)=ξ^{\brk{\lambda,\lambda}+\brk{\pi,\lambda}}v.$$
Then the first property of the lemma combined with Equation \eqref{E:lLambdaW} implies that 
  $\brk{\lambda+\pi,\lambda}\in \ell\Z$ and the second property of the lemma follows.
\end{proof}

\begin{proposition}
  For any $z,z'\in\Zt$ one has $\sigma(z)\in\catd_0$
  $$\sigma(z)\otimes \sigma({z'})=\sigma({z+z'}),\quad\dim_\catn(\sigma(z))=(-1)^{\bar z},\quad θ_{\sigma(z)}=\Id_{\sigma(z)}.$$
  Finally, $\{\sigma(z)\}_{z\in \Zt}$ is a free realization of $\Zt$ in $\catd$.
\end{proposition}
\begin{lemma}[Compatibility]\label{P:compat}
There exists a bilinear map $\psi: \Gr \times \Zt \rightarrow \C^*$
  satisfying the compatibility property of Definition
  \ref{D:6}.
\end{lemma}
\begin{proof}
Let us fix $[\mu] \in \Gr=\h^*/\Lambda_R$ and $z=(\bar z,\lambda)\in \Zt=\Z_2 \times \Lambda_\Zt^0$.  Let $\mu\in \h^*$ be a representative of $[\mu]$.  Let $V\in \catn_{[\mu]}$ and let $v \in V$ and $s \in \sigma(z)$ be two vectors.

As in the last lemma, since the root vectors act trivially on $\sigma(z)$, the only part from the action
of ${\Rm}^{ξ}=\check{{\Rm}}^{ξ}\HR^{ξ}$ contributing to the braiding $c_{\sigma(z),V}$ (or $ c_{V,\sigma(z)}$) is $\HR^{ξ}$.
Then Equation \eqref{E:ActionHR} implies:
\begin{equation}\label{eq:5}  
c_{\sigma(z),V}\circ c_{V,\sigma(z)}(v \otimes s)=ξ^{2{\brk{w(v),w(s)}}}v \otimes s
\end{equation}
where $w(v)$ (resp. $w(s)$) is the weight of the weight vector $v$ (resp. $s$).  Here, by the definition of $\catn_{\mu}$ we have that $w(v)$ is of the form $\mu+\alpha$ for some $\alpha \in\Lambda_R$.
Combining this fact 
and Property (1) of Lemma \ref{L:LambdaZ0} we have that 
\begin{equation}\label{eq:6}
\xi^{\brk{w(v),w(s)}}=\xi^{\brk{\mu,\lambda}}
\end{equation}
is a well defined scalar independent of choice of the vector $v$ in $ V$ and the representative $\mu$ of $[\mu]$.    Thus, we can consider the function  $$\psi([\mu],z):=\xi^{\brk{\mu,\lambda}}.$$
From Equations \eqref{eq:5} and \eqref{eq:6}, we conclude that the function $\psi$ satisfies the compatibility condition, which concludes the proof.
\end{proof}

\subsection{The unimodular m-trace of $\catn$}
From \cite[Corollary 5.6]{GKP3} we know that as $\catn$ is a
locally-finite pivotal $\C$-tensor category which has enough
projectives, it has a unique non trivial m-trace on its ideal of
projective objects if and only if $\catn$ is unimodular.  We use this to prove the following theorem:
\begin{theorem}\label{T:ExTraceCatn}
  There is an unique (up to a global scalar) trace on the ideal of
  projective modules of $\catn$.
\end{theorem}
\begin{proof}
  We need to show that $\catn$ is unimodular.  In other words, we need to show that the projective cover of the trivial module is
  isomorphic to its dual or equivalently, the injective envelope of
  the trivial module is self dual.  We sketch the proof that follows
  the line of \cite{GP13} where a similar result was proved for all (non-super) simple Lie algebras.  Let $V$ be a typical thus projective module.
 Let  $\lambda$ (resp.\ $\lambda+\pi$) be the highest weight (resp.\ lowest weight) of $V$ with associated highest weight vector $v_+$ (resp.\ lowest weight vector $v_-$).
  Let
  $x_+=\prod_{\alpha\in\roots_1^+}E_\alpha\prod_{\alpha\in\roots_0^+}E_\alpha^{\ell-1}$
  and $x_-=\prod_{\alpha\in\roots_1^+}F_\alpha\prod_{\alpha\in\roots_0^+}F_\alpha^{\ell-1}$.
   Then $x_-v_+\in\C^*v_-$ and
  $x_+v_-\in\C^*v_+$.
  
   The module $V\otimes V^*$ has a
  splitting in indecomposable projective modules:
  $$V\otimes V^*\simeq\bigoplus_i P_i.$$
  Among the $P_i$, the injective envelope of the trivial module
  appears exactly once (because
  $\Hom_\catn(\C,V\otimes V^*)\simeq\End_\catn(V)=\C$). Let us say it
  is $P_0$.  Thus, $P_0$ is the unique module among the $P_i$ that
  contains an invariant vector.  Now the vector $v_+\otimes v_+$ is
  alone in his weight space so it is contained in one of the $P_i$ for
  some $i$.  Similarly, $v_-\otimes v_-\in P_j$.  As
  $(V\otimes V^*)^*\simeq V\otimes V^*$ and the dual of a module has
  opposite weights, it comes $P_i^*\simeq P_j$.  Now the proof follows
  from the fact that $x_+(v_-\otimes v_-)$ and $x_-(v_+\otimes v_+)$
  are both non zero invariant vectors showing that $P_i=P_0=P_j$.
\end{proof}
For $V,V'\in\catn$ with $V'$ simple, we denote by $S'(V,V')$ the
scalar endomorphism which is given by the partial trace of the double
braiding of $V$ and $V'$:
\begin{equation}
  \label{eq:Sprime}
  \epsh{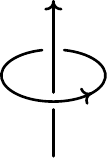}{12ex}\put(-10,-10){\ms{V}}\put(-36,22){\ms{V'}}=S'(V,V')\Id_{V'}.
\end{equation}
Let $\Z[\h^*]$ be the group ring  whose basis $\{\qN\beta\}_{\beta\in\h^*}$ is indexed by the elements of the group $(\h^*,+)$.   For each $\alpha\in\h^*$ consider the ring morphism 
$$
\begin{array}{rcl}
  \varphi_\alpha:\Z[\h^*]&\to&\C \\
  \qN\beta&\mapsto& \xi^{\brk{\alpha,\beta}}.
\end{array}
$$
Any weight module $V$ has a super character $\sch(V)=\sum_\beta d_\beta\qN\beta\in\Z[\h^*]$ where $d_\beta$ is the super dimension of the $\beta$-weight space of $V$.
\begin{lemma}\label{L:Sprime}
  Recall, $\pi=-2(\ell-1)\rho_{\p0}-2\rho_{\p1}$.  Let $V,V'\in\catn$
  with $V'$ simple with highest weight $\lambda$ then
  $$S'(V,V')=\varphi_{2\lambda+\pi}(\sch(V)).$$
  In particular, if $V$ is a simple module of highest weight $\mu$
  which is typical with dimension $D$, then
  $$S'(V,V')=ξ^{2\brk{\lambda+\frac\pi2,\mu+\frac\pi2}}
  \prod_{\alpha\in\roots_{\p0}^+} \dfrac
  {\qn{\ell\brk{\lambda+\frac\pi2,\alpha}}{_ξ}}{\qn{\brk{\lambda+\frac\pi2,\alpha}}{_ξ}}
  \prod_{\alpha\in\roots_{\p1}^+}\qn{{\brk{\lambda+\frac\pi2,\alpha}}}{_ξ}. $$
\end{lemma}
\begin{proof}
The proof will be analogous to the proof of Proposition 2.2 in \cite{GP2}.  In particular, one  only needs to compute the image of the highest weight vector of $V$ (properties of the $R$-matrix make this computation fairly straight forward).  The only difference between the computation in \cite{GP2} and the desired result here is that the two papers use different pivotal structures.  Taking this into account, the computations from Proposition 2.2 in \cite{GP2} imply the first equality from the lemma.  The second equality
follows from the fact that $\sch(V)$ is $[\mu]$ times the super character of the negative Borel, which gives
  $$\sch(V)=[\mu]\prod_{\alpha\in\roots_{\p0}^+}\bp{1+[-\alpha]
    +\cdots+[(1-\ell)\alpha]}\prod_{\alpha\in\roots_{\p1}^+}(1-[-\alpha])\in\Z[\h^*].$$
Then the formula is obtained from 
  $$S'(V,V')=ξ^{\brk{2\lambda+\pi,\mu}}
  \prod_{\alpha\in\roots_{\p0}^+} \dfrac
  {1-ξ^{-\ell\brk{2\lambda+\pi,\alpha}}}{1-ξ^{-\brk{2\lambda+\pi,\alpha}}}
  \prod_{\alpha\in\roots_{\p1}^+}\bp{1-ξ^{-\brk{2\lambda+\pi,\alpha}}}. $$
\end{proof}
\begin{corollary}\label{C:mtraceD}
  The m-trace on the ideal of projective modules can be
  normalized so that the modified dimension of the projective simple
  module $V^\lambda$ is given by
  $$\qd\bp{V^\lambda}=\dfrac
  {\prod_{\alpha\in\roots_{\p0}^+}\qn{\brk{\lambda+\frac\pi2,\alpha}}_ξ}
  {\prod_{\alpha\in\roots^+}\qn{{n_\alpha\brk{\lambda+\frac\pi2,\alpha}}}_ξ}$$
  where $n_\alpha=1$ if $\alpha\in\roots_{\p1}$ and $n_\alpha=\ell$ if
  $\alpha\in\roots_{\p0}$. Furthermore, the module
  $V^{(\ell-1)\rho_0}$ is a simple projective module in $\catn_0$ of dimension $D=\ell^{\card(\roots_{\p0}^-)}2^{\card(\roots_{\p1}^-)}$.
\end{corollary}
\begin{proof} Let $\mt$ be the m-trace given in Theorem \ref{T:ExTraceCatn}.  
Let $\mu$ be a weight with $[\mu]\in \Gr\setminus \XX$.  From Theorem \ref{T:SScatn} we know the module $V^\mu$ is a projective simple module.  If $V$ and $W$ are projective modules then the properties
of the m-trace imply
$$
\qd(V)S'(W,V)=\mt_{V\otimes W}(c_{W,V}\circ c_{V,W}) =\mt_{W\otimes
  V}(c_{V,W}\circ c_{W,V})=\qd(W)S'(V,W).
$$
Let $f(\lambda)$ be the function of $\lambda$ given by 
$$f(\lambda)=ξ^{-2\brk{\lambda+\frac\pi2,\mu+\frac\pi2}}S'(V^\mu,V^\lambda)$$
see Lemma \ref{L:Sprime}.
Then  
\begin{equation}
  \label{eq:S'd}
  \qd(V^\lambda)=\qd(V^\mu)\frac{S'(V^\lambda,V^\mu)}{S'(V^\mu,V^\lambda)}
  =\qd(V^\mu)\frac{f(\mu)}{f(\lambda)}.
\end{equation}
So $\qd(V^\lambda)$ is a constant
divided by $f(\lambda)$ (as $\lambda$ varies this constant stays the
same). By rescaling the m-trace by $(\qd(V^\mu)f(\mu))^{-1}$ we have
$\qd(V^\lambda)=\frac1{f(\lambda)}$ and we obtain the formula given in
the corollary.

For the second statement of the corollary, notice that Lemma
\ref{L:Sprime} implies $S'(V^\lambda, V^{(\ell-1)\rho_0})\neq 0$.
Thus, the open Hopf link with $V=V^\lambda$ and
$V'=V^{(\ell-1)\rho_0}$ given in Equation \ref{eq:Sprime} is a
non-zero morphism.  This morphism defines $V^{(\ell-1)\rho_0} $ as a
retract of
$(V^\lambda)^* \otimes V^{(\ell-1)\rho_0} \otimes V^\lambda$.  Since
$V^\lambda$ is projective then
$(V^\lambda)^* \otimes V^{(\ell-1)\rho_0} \otimes V^\lambda$ is
projective.  Also, since the set of projective objects is an ideal,
then using that $V^{(\ell-1)\rho_0} $ is a retract, it implies that it is in the ideal as well and so it is
projective.

Now since $V^{(\ell-1)\rho_0} $ is a simple projective module, the non
degenerate m-trace is non zero on
$\End_\catd(V^{(\ell-1)\rho_{\p0}})=\C\Id$.  Let
$d_0=\qd(V^{(\ell-1)\rho_{\p0}})\neq 0.$ Then Equation \eqref{eq:S'd} and
Lemma \ref{L:Sprime} imply that for generic $\lambda\in\h^*$,
$$\vp_{2\lambda+\pi}(\sch(V^{(\ell-1)\rho_0}))
=S'(V^{(\ell-1)\rho_0},V^\lambda)=d_0S'(V^\lambda,V^{(\ell-1)\rho_0})\qd(V^\lambda)^{-1}$$
but
$$S'(V^\lambda,V^{(\ell-1)\rho_0})=\vp_{-2\rho_1}(\sch(V^\lambda))=
ξ^{\brk{\lambda,-2\rho_1}}\ell^{\card(\Delta_0^+)}(1-ξ^{n-m})^{\card(\Delta_1^+)}.$$
So, there exists two constants $C,C'$, that does not depends on $\lambda$ such that
\begin{align*}
  \vp_{2\lambda+\pi}&(\sch(V^{(\ell-1)\rho_0}))=\\&=C\,ξ^{2\brk{\lambda,-\rho_1}}\qd(V^\lambda)^{-1}=
  \\&=C'\vp_{2\lambda+\pi}([(\ell-1)\rho_0]\prod_{\alpha\in\roots_{\p0}^+}\bp{1+[-\alpha]
  +\cdots+[(1-\ell)\alpha]}\prod_{\alpha\in\roots_{\p1}^+}(1-[-\alpha])).
\end{align*}
Since this is true for generic $\lambda$,
we have  that the super character 
$$\sch(V^{(\ell-1)\rho_0})=[(\ell-1)\rho_0]\prod_{\alpha\in\roots_{\p0}^+}\bp{1+[-\alpha]
  +\cdots+[(1-\ell)\alpha]}\prod_{\alpha\in\roots_{\p1}^+}(1-[-\alpha])).$$
 To see that its dimension
is $D$, let $\omega=ξ^{\frac{\ell n}{2(m-n)}}$ so that
$\omega^{m-n}=(-1)^n$.  We consider the ring morphism
$\psi:\Z[\Lambda_W]\to\C$ which sends $[\ve_i]\mapsto\omega$
($i=1\cdots m$) and $[\delta_i]\mapsto-\omega$ ($i=1\cdots n$).  (It
is well defined because
$\psi([\str])=\psi([\sum_i\ve_i-\sum_j\delta_j])=
\omega^m(-\omega)^{-n}=1$).
 Then $\psi([\alpha])=1$ if $\alpha\in\Delta_{\wb 0}^+\cup \Delta_{\wb 0}^-$ and $\psi([\alpha])=-1$ if $\alpha\in\Delta_{\wb 1}^+\cup \Delta_{\wb 1}^-$.
 Now $\psi(\sch(V^{(\ell-1)\rho_0}))=D=\ell^{\card(\roots_{\p0}^+)}2^{\card(\roots_{\p1}^+)}$.   Then 
$D=|\psi(\sch(V^{(\ell-1)\rho_0}))|\le \dim_\C(V^{(\ell-1)\rho_0})$.  On the other hand, from the remark in Notation
\ref{N:1} we have $\dim_\C(V^{(\ell-1)\rho_0})\le D$, so
$\dim_\C(V^{(\ell-1)\rho_0})=D$.
\end{proof}
\subsection{The relative modular category of projective modules}
In this subsection, we finalize the construction and show that the
 category $\catn$ is a relative modular category.  
Above, we have already shown that $\catn$ is
pre-modular $\Gr$-category. 
It remains to prove the existence of a modularity parameter (see Definition \ref{D:2}). We start with the following definition. 
 \begin{definition}[Transparent endomorphism]
 Let us fix an object $V \in \catn_0$. A morphism $f \in \End_{\catn}(V)$ is called \emph{transparent} if for any $U,W \in \catn_0$ one has:
 \begin{equation}\label{eq:7} 
 \begin{cases}
c_{V,U}\circ (f \otimes Id_{U}) \circ c_{U,V}=Id_{U} \otimes f\\
c_{W,V}\circ (Id_{W} \otimes f) \circ c_{V,W}=f \otimes Id_{W}. 
 \end{cases}
 \end{equation}
 
 \end{definition}
\begin{lemma}[Factorisation property] \label{L:4} 
Let $V \in \catn$ and $f \in End_{\catn}(V)$ be a transparent morphism. Then, there exist $n\in \N$, $z_{1},...,z_n \in \Zt$ and a set of morphisms $\pi_i \in \Hom_{\catn}\left(V, \sigma (z_i) \right)$ and $\eta_i \in \Hom_{\catn}\left( \sigma (z_i), V \right)$ such that: 
\begin{equation}\label{eq:10}
f= \sum_{i=1}^n \eta_i \circ \pi_i.
\end{equation}
\end{lemma}
\begin{proof}
The proof of this factorisation property is similar to the one presented in Lemma 2.3 of \cite{DGP2} for the case of the unrolled quantum group.
The main idea is to use the transparency property of $f$ presented in Equations \eqref{eq:7} for the modules $V$
and $U=W$, which is the dimension $D$ simple projective module $V^{(\ell-1)\rho_0}$.
Let us represent the modules $V$ and $V^{(\ell-1)\rho_0}$ with their corresponding representation $\rho_{V}$ and $\rho_{V^{(\ell-1)\rho_0}}$, respectively.   If $v_+$ is a highest weight vector of $V^{(\ell-1)\rho_0}$ and $v$ is any weight vector of $V$ then $c_{V^{(\ell-1)\rho_0},V}(v_+ \otimes v)$ is proportional to $v \otimes v_+$ because
 \[
   \rho_{V^{(\ell-1)\rho_0}} \left(\prod_{\alpha \in \Delta^+}E_{\alpha}^{b_\alpha} \right)(v_+) = 0
 \]
for all positive integers $b_\alpha$ whose sum is non zero.\\
Furthermore, $c_{V,V^{(\ell-1)\rho_0}}(f(v) \otimes v_+)$ is proportional to $v_+ \otimes f(v)$ because $f$ is transparent. From Corollary \ref{C:mtraceD} we have
 \[
\left\{ \rho_{V^{(\ell-1)\rho_0}} \left( \prod_{\alpha \in \Delta^+}F_{\alpha}^{b_\alpha}  \right)(v_+) \Biggm| 0 \leq b_\alpha < \ell \text{ if } \alpha\in \roots_{\p0}^+ \text{ and } 0 \leq b_\alpha \leq 1 \text{ if } \alpha\in \roots_{\p1}^+ \right\}  
 \]
 is a basis of $V^{(\ell-1)\rho_0}$. This means that
 \[
  \rho_V \left( \prod_{\alpha \in \Delta^+}E_{\alpha}^{b_\alpha} \right)(f(v)) = 0
 \]
 for every weight vector $v \in V$ and  for all integers $b_\alpha$ (whose sum is strictly positive) such that $0 \leq b_\alpha < \ell$ if $\alpha\in \roots_{\p0}^+$ and $0 \leq b_\alpha \leq 1$ if $\alpha\in \roots_{\p1}^+$.  
 
 Analogously, \[
  \rho_V \left( \prod_{\alpha \in \Delta^+}F_{\alpha}^{b_\alpha} \right)(f(v)) = 0
 \]
 for every weight vector $v \in V$ and  for all integers $b_\alpha$ (whose sum is strictly positive) such that $0 \leq b_\alpha < \ell$ if $\alpha\in \roots_{\p0}^+$ and $0 \leq b_\alpha \leq 1$ if $\alpha\in \roots_{\p1}^+$. 

Now, we look at the weight of the vector $f(v)$. We will use the following relation between the generators of the quantum group from Definition \ref{relqgp}:
\begin{equation}\label{eq:9}
[E_{i},F_{j}]=δ_{i,j}\frac{K_i-K_i^{-1}}{ξ^{d_i}-ξ^{-d_i}}.
\end{equation}
Combining the previous equations
and \eqref{eq:9}, we get:
$$(K_i-K_i^{-1})\left( f(v) \right)=0.$$
This is equivalent to:
$$(K_i^{2}-1)\left( f(v) \right)=0, \forall i\in \{1,...,\rk\}.$$
The previous relation shows that $$2\brk{w(f(v)),\alpha_i}\in \ell \Z, \ \ \ \forall i\in \{1,...,\rk\}, \forall v\in V$$ 
here $w(f(v))$ is the weight of the weight vector $f(v)$.

We conclude that each weight vector from the image of $f$ generates a
one dimensional module whose weight is in $\Lambda_\Zt$ (see
Definition \ref{D:3}). This means that the vector $f(v)$ generates a
submodule which is isomorphic to $\sigma(z_v)$ for $z_v\in \Zt$. This
procedure applied to all the weight vectors of $V$, shows that there
exists $n \in \N$ and some $z_i \in \Lambda_{\Zt}$ such that the image
of $f$ is decomposed as below:
\begin{equation}
Im(f)\simeq \bigoplus _{i=1}^{n} \sigma (z_i) \subseteq V.
\end{equation}
Let us consider the corresponding projection induced by $f$ and
inclusion maps:
\begin{equation}
\begin{aligned}
& \pi_i: V\rightarrow \sigma (z_i)\\
& \eta_i: \sigma (z_i) \rightarrow V. 
\end{aligned}
\end{equation}
This tells us that we have the desired decomposition from Equation
\eqref{eq:10} and concludes the proof.
\end{proof}
\begin{proposition}[Modularity]
  There exists a modularity parameter $ζ_{\Omega}$ such that corresponding
  condition of Definition \ref{D:2} is true in the category $\catn$.
\end{proposition}
\begin{proof}
  The existence of $ζ_{\Omega}\in\C$ follows exactly as the proof of
  modularity property from \cite{DGP2}. The key ingredients are the
  factorisation property for transparent endomorphisms, as we showed
  in Lemma \ref{L:4} together with certain general properties of
  modified traces.

  Now to show that $ζ_{\Omega}$ is non zero, we compute the m-trace of both
  side of \eqref{eq:mod} when $V_i=V_j=V^\lambda$ and
  $\Omega_h=\sum_\mu\qd(V^{\mu})V^{\mu}$.  This gives
  \begin{align*}
    ζ_{\Omega}&=\sum_\mu\qd(V^{\mu})\qd(V^\lambda)S'(V^{\mu},V^\lambda)S'(\bp{V^\lambda}^*,V^{\mu})=\sum_\mu1\neq0
  \end{align*}
  where the second equality follows from
  $S'(V^{\mu},V^\lambda)=ξ^{2\brk{\lambda+\frac\pi2,\mu+\frac\pi2}}\qd(V^\lambda)^{-1}$
    (see Lemma \ref{L:Sprime} and Corollary \ref{C:mtraceD}) and $\bp{V^\lambda}^*\simeq V^{-\pi-\lambda}$.
\end{proof}
Putting together all these properties, we conclude.
\begin{theorem}\label{THEOREM1b}
$\catn$ is a modular $\Gr$-category relative to $(\Zt,\XX)$.
\end{theorem}

\section{Perturbative modules}
In this section, we study the category of perturbatives modules $\catp$ (see definitions \ref{D:perturb} and \ref{D:PertMod}).
\begin{remark}\label{R:piRho}
  The category $\catp$ is also the full subcategory of $\catd$ form by
  module where $K_{2\ell\rho_0}$ acts trivially.  Indeed, for
  $i=1\ldots r$, $\brk{2\ell\rho_0,w_i}$ is in $\ell\Z$ and
  $\brk{2\ell\rho_0,w_m}=0$ so for any perturbative weight $\lambda$,
  $ξ^\brk{2\ell\rho_0,\lambda}=1$.  In particular the pivotal structure
  in $\catp$ given by $K_\pi$ coincides with the specialisation of the pivotal structure
  of $U_q(\frakg)$-mod given by $K_{2\rho}$.
\end{remark}
\subsection{Simple typical modules}
As we will explain, the characters and fusion rules for typical perturbative  modules 
reduce to the study of $\Ux(\frakg_0)$-modules at root of unity. 
This is in
general a difficult problem but it is known that simple highest
weight modules with the highest weight in an alcove $C^\le$ are
specializations of $\Ux(\frakg_0)$-modules for generic $q$.

\begin{definition}[Alcove]
Consider the following sums of even simple roots $$\alpha_{L_1}=\sum_{i=1}^{m-1}\alpha_i \;\;\text{ and }\;\; \alpha_{L_2}=\sum_{i=m+1}^{m+n-1}\alpha_i.$$
Note $\alpha_{L_1}$ (resp. $\alpha_{L_2}$) is also the  highest weight of the adjoint representation of $\sll(m)\subset\g_0$ (resp. $\sll(n)\subset\g_0$).
  Let us consider the following subsets of weights:
  $$C=\{\w ac\in \h^*:(c_1,\ldots,c_{r-1})\in\N^{r-1},a\in\C\}$$
  $$C^\le=\{\w ac\in C:\brk{\w ac+\rho_0,\alpha_{L_1}}\le \ell
  \text{ and }\brk{\w ac+\rho_0,-\alpha_{L_2}}\le \ell\}$$
$$C^<=\{\w ac\in C:\brk{\w ac+\rho_0,\alpha_{L_1}}< \ell\text{ and
}\brk{\w ac+\rho_0,-\alpha_{L_2}}< \ell\}.$$
 In particular, for a weight
$\w ac= c_1w_1+\cdots+c_{m-1}w_{m-1}+a 
w_m+c'_1w_{m+1}+\cdots+c'_{n-1}w_{{\rk}}$, we have
\begin{equation}\label{eq:long-root}
  \brk{\w ac+\rho_0,\alpha_{L_1}}=\sum_{k=1}^{m-1}(c_k+1)\text{ and
  }\brk{\w ac+\rho_0,-\alpha_{L_2}}=\sum_{k=1}^{n-1}(c'_k+1).
\end{equation}
\end{definition}
Using this, we obtain the following description for the alcove:
\begin{equation}\label{D:7}  
\begin{aligned}
 C^\le=\Big\{\w ac\in C \mid  \sum_{k=1}^{m-1}(c_k+1)\leq  \ell \text{ and } \sum_{k=1}^{n-1}(c'_k+1) \leq  \ell \Big\}\\
C^{<}=\Big\{\w ac\in C \mid  \sum_{k=1}^{m-1}(c_k+1)<  \ell \text{ and } \sum_{k=1}^{n-1}(c'_k+1) <  \ell \Big\}
\end{aligned}
\end{equation}

The following theorem is a consequence of \cite{APK}.
\begin{theorem}\label{T:Andersen}
  Assume $\lambda\in C^\le$ then the simple $U_ξ\g_0$-module
  $V^\lambda_ξ$ with highest weight $\lambda$ is obtained by
  specialization of the $U_q\g_0$ simple module $V^\lambda_q$ of
  highest weight $\lambda$.  In particular, they have the same
  dimension and the same character.
\end{theorem}
\begin{proof} The specialisation of a simple module at generic $q$ is
  called a Weyl module.  In \cite[Proposition 11.2.5]{CP95},
  inequalities similar to those defining the alcove are defined for
  the quantum groups $U_\xi\sll_m$ or $U_\xi\sll_n$.  As a
  consequence, the simple $U_\xi\g_0$-modules, with highest weight in
  the alcove are the tensor product of a $U_\xi\sll_m$ Weyl module
  with a $U_\xi\sll_n$ Weyl module.  Note, that in \cite{CP95}, Chari
  and Presley, work with a larger algebra that contains the
  $\ell$\textsuperscript{th} divided powers of the generator but they
  show in Proposition 11.2.10 that these $\ell$\textsuperscript{th}
  divided powers act by $0$ on these modules and so they remain simple
  as a $U_\xi\g_0$-modules.
\end{proof}

\begin{remark}\label{R:positiveRootEv} Recall $\roots^+_{\p 0}=\{ε_i-ε_j,\,1\leq i<j\leq m\}\cup
\{δ_i-δ_j,\,1\leq i<j\leq n\}$.
We will need the following inequalities in the proof of Theorem \ref{C:22}. For $1\leq i<j\leq m$ and $1\leq s<t\leq n$ a direct calculation shows 
\begin{align*}
  \brk{\w ac+\rho,ε_i-ε_j}  &= j-i+\sum_{k=i}^{j-1}c_k \leq m-1 + \sum_{k=i}^{m-1}c_k\\
     \brk{\w ac+\rho,-δ_s+δ_t}    & =-s+t+\sum_{k=s}^{t-1}c_k' \leq n-1 +\sum_{k=1}^{n-1}c_k'
     \end{align*}
    where these sums are strictly greater than zero. Thus,
     If  $\w ac \in C^<$ then 
     $$ \brk{\w ac+\rho,ε_i-ε_j},  \brk{\w ac+\rho,-δ_s+δ_t}\in \{1,...,\ell-1\}.$$  
     Moreover, if $\w ac \in C^\le\setminus C^{<}$ then
     $$
     \brk{\w ac+\rho,ε_1-ε_m}=\brk{\w ac+\rho,\alpha_{L_1}}=\ell\text{ or }  \brk{\w ac+\rho,δ_n-δ_1}=\brk{\w ac+\rho,-\alpha_{L_2}}=\ell.
     $$
\end{remark}

\subsection{Grading on the category $\catp$}\label{S:modiftrace1}
Recall the category of perturbative modules $\catp$ given in Definition \ref{D:PertMod}.
Using the grading defined in Definition \ref{D:3} on the category of  weight modules $\catn$, we will induce a grading on  $\catp$, keeping in mind the fact that we are working at roots of unity.
\begin{definition}[Grading]
  Let us consider the abelian group $\Gp=(\C/ \Z,+)$.
  For $ \bar {z} \in \Gp$, let $\catpz$ be the full subcategory of
  $\catp$ whose objects have the following weight:
  $$Ob(\catpz)=\{ V \in \catp \mid  K_m^\ell \curvearrowright V =
  \xi^{z\ell} \cdot \Id_{V}, \text{ for any representative } z \in \C \text{  of }
  \bar{z}  \}.$$
\end{definition}
\begin{remark}
  If we assume $m>n$, then $\Lambda_W/\Lambda_R$ is cyclic of order
$m-n$ generated by the class of $w_1$.  Also, since $K_m$ acts by $1$ on a weight vector of weight $w_1$, we have
$$\catpz=\bigoplus_{i=0}^{|m-n|-1}\catn_{[\lambda^0_z+iw_1]}.$$
\end{remark}
We are interested in studying certain situations where the decomposition of the tensor product is semi-simple. For this purpose, the typical weights will play an important role. We begin with a remark concerning typical weights in the perturbative setting.
\begin{remark}\label{R:perturbativeWeights}
 We notice that for the case of perturbative weights, we have $c\in\Z^{m+n-2}$. Then, if $a \notin \frac12 \Z$, the condition given in Zhang's Theorem from Equation~\eqref{eq:typ'} is satisfied and so $\lambda_a^c$ is typical. 
This shows that:
$$\C\setminus\frac12\Z\subset \TT_c \text{ for all } c \in \N^{{\rk}-1}$$
where $\TT_c$ is defined in  Notation \ref{N:Ctypical}.
\end{remark}
This remark shows us that away from half integers we obtain typical weights. This motivates the following definition. 
\begin{definition}[Critical set $\Xp$]\label{D:8}  
Let us consider: $$ \Xp:=  \left( \frac12 \Z \right) / \Z   \subseteq \Gp=\C / \Z.$$
\end{definition}

\begin{notation}\ 
  \begin{itemize}
  \item Let $\lambda^0_{a}\in\h^* $ be the weight corresponding to
$(\bar0,a)\in \N^{\rk-1}\times\C$. 
  \item For $\lambda\in\h^*$, let $V_ξ^\lambda$ (resp.\ $V_q^\lambda$) denotes the unique (up to isomorphism) simple highest weight $U^{H}_ξ(\g)$-module (resp.\ $U_q(\g)$-module) with highest weight $\lambda$.
  \end{itemize}
\end{notation}
\subsection{The m-trace on the ideal generated by a typical perturbative module}\label{S:modiftrace2}
To define the m-trace, we need to understand the decomposition of certain tensor products. We will use the following general result, from Lemma 3.4.4. of \cite{AG17}:

\begin{lemma}\label{L:3}  Let $\cat$ be a category of finite dimensional $B$-modules where $\kk$ is a field and $B$ is a $\kk$-algebra.  Let  $V$  be an object in $\cat$.  Let $V_1,...,V_n$ be simple submodules of $V$ such that:
\begin{enumerate}
\item  $V_i$ is not isomorphic to $ V_j$, for all $ i,j \in \{1,..,n \}, i\not=j$,
\item  $\dim_\kk(V_1)+...+\dim_\kk(V_n)=\dim_\kk(V)$ where $\dim_\kk$ is the vector space dimension.  
\end{enumerate}
Then $V=V_1 \oplus... \oplus V_n.$
\end{lemma}
Using this result, we deduce the following decomposition.
\begin{lemma}\label{L:21}
  Let $a,b\in\C$ such that $\wb{a+b}\notin\Xp$. Then
  $V^{\lambda^0_a}_ξ\otimes V^{\lambda^0_b}_ξ$ is a direct sum of
   typical modules with no multiplicity. Furthermore, $V^{\lambda^0_{a+b}}_ξ$ is a direct summand of this module.
\end{lemma}
\begin{proof}
  Let $\K$ be the
  field of fractions of $A=\Z[q^{\pm1},\qn1_q^{-1}]$ and consider the
  tensor product $V^{\lambda^0_a}_q\otimes_A
  V^{\lambda^0_b}_q$,
  which is a representation of
  the integral version of $U_q(\g)$ over the ring $A$. This is an
  $A$-lattice in the vector space
  $V^{\lambda^0_a}_q\otimes_A V^{\lambda^0_b}_q\otimes_A \K$. For this
  version, the semi-simple decomposition of
  $V^{\lambda^0_a}_q\otimes_A V^{\lambda^0_b}_q\otimes_A \K$ is known and
  given in \cite{GP2}.  More precisely, this a certain direct sum of
  typical modules which are typical envelopes of $\g_0$-modules. These
  $\g_0$ modules that occur in the decomposition, are indexed by the
  young diagrams
  bounded by a $m\times n$ rectangle.
  
  More specifically, the young diagrams
  are prescribed by
  $m$-tuples:
  $$\lambda=(\lambda_1\ge\lambda_2\ge\cdots\ge\lambda_m\ge0 ) \text {
    such that } \lambda_1\le n.$$ Let us fix $\lambda$. Then, let us
  construct the following two partitions:
\begin{equation}  
\begin{cases}
\text{ the complementary partition }\hat \lambda_i= n-\lambda_{m+1-i} \text{ for } i=1\cdots m\\
\text{ the conjugate partition of } \hat{\lambda} \text{ denoted by } \mu=(\mu_1,\ldots,\mu_n).  
\end{cases}
\end{equation}  
Then the highest weight of the $\g_0$ module associated to $\lambda$ is given by
  $$c_\lambda=(\lambda_1-\lambda_2,\ldots,\lambda_{m-1}-\lambda_m,a_\lambda,
  \mu_1-\mu_2,\ldots,\mu_{n-1}-\mu_n)$$
  where $a_\lambda=c_\lambda(H_m)=z-mn+\sum_i\lambda_i$.  In particular, using Equation~\eqref{eq:long-root} we have:
  \begin{equation}
    \label{eq:alcoveyoung}
    \brk{c_\lambda+\rho_0,\alpha_{L_1}}=\lambda_1-\lambda_m+m-1\text{ and
  }\brk{c_\lambda+\rho_0,\alpha_{L_2}}=\mu_1-\mu_n+n-1.
  \end{equation}
We notice that both scalar products are lower than $m+n-1= {\rk}$, so the weight $c_\lambda$ belongs to the alcove $ C^\le$.

Now, let us consider the highest weight vectors $\bp{v_\lambda^\K}_\lambda$ of
  these modules in $V^{\lambda^0_a}_q\otimes_A V^{\lambda^0_b}_q\otimes_A \K$.  We can
  rescale these vectors to get primitive vectors
  $\bp{v_\lambda^A}_\lambda$ in the $A$-lattice
  $V^{\lambda^0_a}_q\otimes_A V^{\lambda^0_b}_q$ and in turn, these
  vectors specialize at $q=ξ$ to non zero highest weight vectors
  associated to distinct weights.  This implies that
  $V^{\lambda^0_a}_ξ\otimes_\C V^{\lambda^0_b}_ξ$ contains the simple
  modules $(V^{c_\lambda}_ξ)_\lambda$.  The sum of their dimension is 
  $$\sum_\lambda\dim_\C\bp{V^{c_\lambda}_ξ}
  =\sum_\lambda\dim_\K\bp{V^{c_\lambda}_q}
  =\dim_A\bp{V^{\lambda^0_a}_q\otimes_A V^{\lambda^0_b}_q}
  =\dim_\C\bp{V^{\lambda^0_a}_ξ\otimes_\C V^{\lambda^0_b}_ξ}$$
  and the theorem follows from Lemma \ref{L:3}.
\end{proof}
In the next part, we will use the following definition.
\begin{definition}
Given an object $V \in \catp$, the ideal generated by $V$ in $\catp$, denoted by $\ideal_V$, is the full subcategory consisting of all retracts (i.e.\ direct summands) of objects of the form $V\otimes W$ where $W\in \catp$.  
\end{definition}
\begin{theorem}\label{C:22}
 \begin{enumerate}
\item If  $V^{\lambda^0_a}_ξ$ and $V^{\lambda^0_b}_ξ$  are typical modules with degree $\wb a, \wb b\in \Gp \setminus \Xp$ then they both
  generate the same ideal $\idealp$ in $\catp$. 
  \item There
  exists a unique (up to multiplication by an element of $\C^*$) m-trace $\mtp$ on $\idealp$.
  \item
  Any typical module $V_ξ^{\lambda_a^c}$ in $C^\le$  with degree $\wb a\in \Gp \setminus \Xp$ belongs to  $\idealp$, and has non-zero modified dimension  if and  only if it is in $C^<$.
  
\end{enumerate}
\end{theorem}
\begin{proof}

Let $a\in \C$ such that $\wb a \notin \Xp$.   There exist $z\in \C\setminus \frac14\Z$ 
such that $a-z \in \C\setminus \frac12\Z$. Set $b:=a-z$.  Then from Lemma \ref{L:21} we have that  $V^{\lambda^0_a}_ξ$ is a direct summand of $V^{\lambda^0_{z}}_ξ\otimes V^{\lambda^0_b}_ξ$.  Thus, $V^{\lambda^0_a}_ξ$ is in $\ideal_{V^{\lambda^0_z}_ξ}$.  Set $\idealp:=\ideal_{V^{\lambda^0_z}_ξ}$.

Thus, to prove the first two properties of the theorem it is enough to prove the following fact:
\begin{fact}\label{F:IdealsEqUniqueTrace}
The ideals  $\idealp$ and $\ideal_{V^{\lambda^0_a}_ξ}$ are equal and there is a unique trace on this ideal.  
\end{fact}
We prove this fact in several steps as follows.  Since
$z\in \C \setminus \frac14 \Z$ then Lemma \ref{L:21} implies
$V^{\lambda^0_z}_ξ\otimes V^{\lambda^0_z}_ξ$ is a direct sum of typical
modules with no multiplicity.  Then by Remark 3.3.5 in \cite{GKP1}
(which is a consequence of Lemma 3.3.4 in the same paper) we have that
$V^{\lambda^0_z}_ξ$ is \emph{ambidextrous}. Let $d_0$ be an element of
$\C$.  The general theory of \cite{GKP1} (in particular Corollary
3.3.3) implies that each simple ambidextrous object $V$ determines a
unique non-zero m-trace on the ideal $\ideal_{V}$ such that
the modified dimension of $V$ is $d_0$.  Denote the m-trace
corresponding to $V^{{\lambda^0_z}}_ξ$ by
$\mtp=\{\mtp_V\}_{V\in\idealp}$ and let
$\qdp=\{\qdp(V)=\mtp_V(\Id_V)\}_V$ be its associated modified dimension,
normalized with
$\qdp(V^{{\lambda^0_z}}_ξ)=d_0:=\prod_{\alpha\in\roots_{\p1}^+}\frac1{\qn{\brk{\lambda^0_z
      +\rho,\alpha}}}.$
For simple modules $V, W \in \idealp$, we consider the scalar
$S'(W,V)\in \C$ given in Equation \eqref{eq:Sprime}.  The properties
of the m-trace imply
$$
\qdp(V)S'(W,V)=\mtp_{V\otimes W}(c_{W,V}\circ c_{V,W}) =\mtp_{W\otimes
  V}(c_{V,W}\circ c_{W,V})=\qdp(W)S'(V,W).
$$
Using Lemma \ref{L:Sprime}, Remark \ref{R:piRho} and the fact that modules in the alcove
have the same character as their $U_q(\frakg)$ versions (see
\cite{GP2}), we then have for a typical weight $\lambda_a^c$ in the
alcove:
\begin{equation}
  \label{eq:qdp}
  \qdp\bp{V_\xi^{\lambda_a^c}}=
  d_0\frac{S'(V_\xi^{\lambda_a^c},V^{{\lambda^0_z}}_ξ)}
  {S'(V^{{\lambda^0_z}}_ξ,V_\xi^{\lambda_a^c})}=
  \prod_{\alpha\in\roots_{\p0}^+} \frac {\qn{\brk{\lambda_a^c+\rho,\alpha}}_ξ}
  {\qn{\brk{\rho,\alpha}}_ξ}
  \prod_{\alpha\in\roots_{\p1}^+}\frac 1{\qn{\brk{\lambda_a^c+\rho,\alpha}}}_ξ. 
\end{equation}
Remark that Equation \eqref{eq:typ} and $\ell\ge {\rk}$ imply that the denominators of $\qdp$ are non-zero complex numbers.  Moreover, this formula and Remark \ref{R:positiveRootEv} imply that $  \qdp(V_ξ^{\lambda_a^c}) $ is non-zero if $\lambda_a^c\in C^<$ and zero if $\lambda_a^c\in C^\le\setminus C^{<}$.  Thus, we have proved the last statement of the theorem.  

From Theorem 4.2.1 of \cite{GKP1} if $V\in \idealp$ and  $\qdp(V) \neq 0$ then $\idealp= \ideal_{V} $.  Since $V^{\lambda^0_a}_ξ\in C^<$, Fact \ref{F:IdealsEqUniqueTrace} follows.

\end{proof}
We will use later the following corollary:
\begin{corollary}\label{C:translate-V0}
  Let $a,b\in\C$ such that $\wb a, \wb b\notin \Xp$, then $V^{\lambda^0_a}_ξ$ is a
  direct summand of
  $V^{\lambda^0_a}_ξ\otimes V^{\lambda^0_b}_ξ\otimes
  (V^{\lambda^0_b}_ξ)^*$.
  Furthermore, we have that:
  $$(V^{\lambda^0_b}_ξ)^*\simeq
  V_ξ^{{2\rho_1+\lambda^0_{-b}}}\otimes\wb\C^{\otimes mn}.$$
\end{corollary}
\begin{proof}
  The first statement follows directly from Theorem 1.4.1 of
  \cite{GKP1}.  The second follows from looking at the highest weight
  vectors of the simple modules considered in the statement.  Indeed
  the lowest weight vector $v_-$ of $V^{\lambda^0_b}_ξ$ is obtained by
  acting with all odd root vectors on its highest weight vector.  The
  weight of $v_-$ is then the opposite of the one of the highest weight
  vector of $(V^{\lambda^0_b}_ξ)^*$ and they have the same parity.
\end{proof}

\section{The relative pre-modular category of perturbative modules}\label{S:perturbative}

\subsection{The subcategory $\catpb$}
In this subsection, we will construct a subcategory of $\catp$, generated by certain objects. We use the following definition.
\begin{definition}[Category generated by a set]
  If $A$ is a set of objects in $\catp$ then we define the category
  generated by $A$ as the full subcategory of $\catp$ which has as
  objects, all
  direct sums of retracts
  of all tensor products of the form:
$$X_{1}\otimes X_{2}\otimes ...\otimes X_{p}\quad\text{ where }X_i\in A\cup A^*.$$
\end{definition}
We will use the following modules to generate the desired subcategory.
\begin{definition}[Basic modules]\ 
\begin{enumerate}
\item Let $ \bar{\C}$ be the 1-dimensional odd module with trivial action.
\item Let $\mathsf{v}$ be the standard $m+n$ dimensional module in $\catp$.
\item Let $\ve$ be the 1-dimensional module where $H_m$ acts by
  $\ell$, all the $K_i$ act by $1$ and all other generators act by
  zero.
\end{enumerate}
\end{definition}
\begin{definition}[Subcategory generated by standard modules $\catpb$]\label{D:SubcatPB}
Let us consider the subcategory $\catpb$ of $\catp$ generated by the following set: 
$$\left\{\bar{\C}, \ve, \mathsf{v}, V^{\lambda^0_z}_ξ \; \Big| \; z\in \C \text{ with } \bar z\in (\C/\Z)\setminus\Xp  \right\}.$$
\end{definition}

\begin{remark}[m-trace]
  The m-trace $\mtp$ on $\idealp\subset
  \catp$
  defined in Theorem \ref{C:22}, induces a m-trace defined on the
  ideal $\ideal_{V^{\lambda^0_a}_ξ}$ in $\catpb$, where
  $\wb a\in \Gp \setminus \Xp$.  We still denote this trace by $\mtp$
  and the ideal by $\idealpb=\ideal_{V^{\lambda^0_a}_ξ}$.
\end{remark}
\begin{definition}\label{D:4}\
  \begin{enumerate}
  \item A module is \emph{semi-simple typical} if it is a direct sum of 
typical modules $V^{\lambda}_ξ$ or $V^{\lambda}_ξ\otimes\bar{\C}$  where $\lambda \in C^\le$.
\item An object $V\in \catpb$ is said to be \emph{negligible} if
$V\in \idealpb $ and $\mtp_V(f)=0$ for all $f\in \End_{\catpb}(V)$.  
  \end{enumerate}
\end{definition}
We have the following two easy lemmas.
\begin{lemma}\label{L:NegIdeal}
The set of negligible objects form an ideal.  
\end{lemma}
\begin{proof} Let $V,W\in \catp$ with $V$ negligible.  The partial trace property implies that  $V\otimes W$ is  negligible.  On the other hand, suppose $W$ is a retract of $V$, i.e. there exists  $f:W\to V$ and $g:V\to W$ such that $fg=\Id_W$.  If $h\in \End_\catp(W)$ then $$\mtp_W(h) =\mtp_W(\Id_W h) =\mtp_W(fg h) =\mtp_V(ghf)=0.$$
\end{proof}
\begin{lemma} \label{L:simpleNeg} If $V\in\idealpb$ such that $V$ is
  simple and $\qdp(V)=0$ then $V$ is negligible.
\end{lemma}
\begin{proof}
  Since $V$ is simple, $\End_{\catp}(V)=\C \Id_V$ and if
  $f\in \End_{\catp}(V)$ we will define the scalar $\brk{f}\in \C$ as the
  solution to the equation $f=\brk{f} \Id_V$.  If $\qdp(V)=0$ then by definition $\mtp_V(\Id_V)=0$.  Thus,
$$\mtp_V(f)=\mtp_V(\brk{f} \Id_V)=\brk{f} \mtp_V(\Id_V)=0.$$
\end{proof}
\begin{theorem}[Structural description of $\catpb$]\label{T:1}
The $\Gp=\C/\Z$ grading on the category $\catp$ induces a $\Gp$-grading on $\catpb$.
If $a\in \C$ with $\wb{a}\notin \Xp$, then  the objects of $\catpb_{\wb{a}}$ are modules of the form
\begin{equation}\label{E:RTN}
T\oplus  N
\end{equation}
where $T$ is semi-simple typical, $N$ is negligible and both are homogeneous with grading $\wb a$.  Moreover, every  typical module $V_ξ^{\w ac}$ in the alcove $C^\le$  is an object of $\catpb$.
\end{theorem}

\begin{proof}

To prove the theorem we will use the the following three facts, which will be proved below.  In what follows, $a,b$ will be elements of $ \C$ and $\wb a, \wb b $ will represent their classes in $\C/\Z$.
\begin{fact}\label{F:Vlambdaa}
  If $V\in \catp$ is a module of the form given in Equation \eqref{E:RTN} then
  $V\otimes \mathsf{v}$
  and $V\otimes \mathsf{v}^*$ are also of this form.
\end{fact}
\begin{fact}\label{F:TypRet}
Let $V_ξ^{\w ac}\in \catp_{\wb a}$ be a typical module in the alcove $C^\le$ for $a\in \C$ with $\wb a\notin  \Xp$.
   Then $V_ξ^{\w ac}$ is a direct summand of
  $V^{\lambda^0_{a+t}}_ξ \otimes \mathsf{v}^{\otimes s}\otimes
  \bar{\C}^{\otimes u}$ for some  $(s,t,u)\in \N\times\Z\times\{0,1\}$.
\end{fact}
\begin{fact}\label{F:TensorVaVb}
If $V\in \catp_{\wb a}$ is a module of the form given in Equation \eqref{E:RTN} and if $V^{\lambda^0_b}_ξ\in \catp_{\wb b}$ such that $\wb a+\wb b\notin \Xp $ then $V^{\lambda^0_b}_ξ\otimes V$ is also of form given in Equation~\eqref{E:RTN}.
\end{fact}
Now notice that Fact \ref{F:TypRet} implies that every typical module
in the alcove with degree outside of $\Xp $ is
in $\catpb$.

On the other hand, let  
$$
V^{\lambda^0_{a_1}}_ξ \otimes V^{\lambda^0_{a_2}}_ξ \otimes...\otimes V^{\lambda^0_{a_k}}_ξ 
\otimes \mathsf{v}^{\otimes s}\otimes (\mathsf{v}^*)^{\otimes s'}\otimes \bar{\C}^{\otimes t}\otimes \ve^{\otimes u}
$$
be a module in $\catpb_{\wb a}$ where $s,s',t, u\in \N$ and $a_i\in \C$
such that $\wb a:=\sum_{i=1}^k \wb a_i \notin \Xp$.  To prove the
theorem we need to show that this module has the form given in
Equation \eqref{E:RTN}.
First, remark that Lemma \ref{L:NegIdeal} ensures that duals and
retracts of a semi-simple typical (resp. negligible) are semi-simple
typical (resp. negligible).  We now use induction on $k$.  Since
$V^{\lambda^0_{a_1}}_ξ \otimes \bar{\C}^{\otimes t}\otimes
\ve^{\otimes u}$ is of the form given in Equation \eqref{E:RTN} then
Fact \ref{F:Vlambdaa} implies
$V^{\lambda^0_{a_1}}_ξ \otimes W$ is also of this form where we set
$$W=\mathsf{v}^{\otimes s}\otimes
(\mathsf{v}^*)^{\otimes s'}\otimes \bar{\C}^{\otimes t}\otimes
\ve^{\otimes u}.$$  This proves the case $k=1$.
 
 Now, let us suppose that the statement is true for $k-1$.
  Then, there exists $z\in\C$
  such that $\wb z\not\in\Xp$ and $\sum_{i=j}^k \wb a_i+\wb z\not\in\Xp$ for
$j\in \{1,2,3,k\}$.  Now, Corollary
 \ref{C:translate-V0} implies $V^{\lambda^0_{a_k}}_ξ$
 is a direct summand of $V^{\lambda^0_{a_k}}_ξ\otimes
 V^{\lambda^0_{z}}_ξ\otimes V_ξ^{{2\rho_{\p1}+\lambda^0_{-z}}}\otimes\wb\C^{\otimes mn}$ and so it follows that $
 V^{\lambda^0_{a_1}}_ξ \otimes V^{\lambda^0_{a_2}}_ξ \otimes...\otimes
 V^{\lambda^0_{a_k}}_ξ \otimes W\otimes\wb\C^{\otimes mn}
 $ is a direct summand of 
 \begin{equation}\label{E:Vbaaab}
 V^{\lambda^0_{z'}}_ξ\otimes
 V^{\lambda^0_{a_1}}_ξ \otimes V^{\lambda^0_{a_2}}_ξ \otimes...\otimes
 V^{\lambda^0_{a_k}}_ξ \otimes V^{\lambda^0_{z}}_ξ \otimes
 W\otimes\wb\C^{\otimes mn}
\end{equation} 
where $z'=-z-(m-n)$ as $2\rho_{\p1}=-(m-n)w_m$. We will show the module in Equation \eqref{E:Vbaaab} is of the form given in Equation \eqref{E:RTN}.  

This will conclude the induction step because
direct summands of modules of the form $T\oplus  N$ are also of this form.  The choice of $z$ guaranties that $\sum_{i=3}^k \wb a_i+\wb z\not\in\Xp$ and the induction hypothesis says 
$$
V^{\lambda^0_{a_3}}_ξ \otimes V^{\lambda^0_{a_4}}_ξ \otimes...\otimes
V^{\lambda^0_{a_k}}_ξ \otimes V^{\lambda^0_{z}}_ξ
\otimes W\otimes\wb\C^{\otimes mn}
$$
 is of the form given in Equation \eqref{E:RTN}.
Then 
   Fact \ref{F:TensorVaVb} with $a=\sum_{i=3}^k a_i+z$ and $b=a_2$ implies that 
   $$
V^{\lambda^0_{a_2}}_ξ \otimes V^{\lambda^0_{a_3}}_ξ \otimes...\otimes
 V^{\lambda^0_{a_k}}_ξ \otimes V^{\lambda^0_{z}}_ξ \otimes
W \otimes\wb\C^{\otimes mn}  $$
    is of the form given in Equation \eqref{E:RTN}.
    Similarly,  applying Fact \ref{F:TensorVaVb} two more times (with $b=a_1$ and $b=z'$, respectively) we see the module of Equation \eqref{E:Vbaaab} is of the  form given in Equation \eqref{E:RTN}.

In the next part, we will prove three main facts listed above.  

{\bf Fact \ref{F:Vlambdaa} } To prove the first fact,
since $V\otimes\mathsf{v}^*\simeq (V^*\otimes\mathsf{v})^*$,
it suffices to show that $V^{\lambda_{a}^c}_ξ\otimes \mathsf{v}$ is of
the desired form where $V^{\lambda_{a}^c}_ξ$ is a typical module in
the alcove.  Since $V^{\lambda_{a}^c}_ξ$ is simple it corresponds to a
typical module $V$ of the Lie superalgebra $\slmn$ (when $q$ is
generic, the module $V$ deforms to a simple $U_{q}\slmn$-module
$V_q$).  Recall that a typical $\slmn$-module $V$ is determined by a
triple $(\lambda_m, \lambda_n,a)$ where $\lambda_m$ (resp.\
$\lambda_n$) is a Young diagram corresponds to $\sll(m)$ (resp.\
$\sll(n)$) and $a$ is a complex number.
Indeed the
character of $V$ has the form
$\chi_1z^as_{\lambda_m}(x)s_{\lambda_n}(y)$ where the $s$ are the
Schur polynomials, the variables $x$ correspond to $\sll(m)$,
variables $y$ correspond to $\sll(n)$ and $z=\prod_i x_i=\prod_j y_j$
correspond to $w_m$ and can have complex exponent.  Also, the standard
$m+n$-dimensional module $ \mathsf{v}$ has the character
$s_1(x)+s_1(y)$ where $s_1$ is the Schur polynomial associated to the
one box Young diagram (this polynomial is just the sum of variables).
Thus the character of $V\otimes \mathsf{v}$ is
$\chi_1z^as_{\lambda_m}(x)s_{\lambda_n}(y)(s_1(x)+s_1(y))$ which
decomposes by Pieri's formula as the sum of the characters of typical
representations corresponding to triples of two Young diagrams and a
complex number where a single box is added to
$(\lambda_m, \lambda_n)$.  When $q$ is generic, this splitting
describes the decomposition of $V_q\otimes \mathsf{v}$ into a
multiplicity free direct sum of typical $U_{q}\slmn$-modules.  So
$V^{\lambda_{a}^c}_ξ\otimes \mathsf{v}$ contains highest weight
submodules corresponding to adding a single box to
$(\lambda_m, \lambda_n)$ and possibly adding an integer to $a$.  Adding
a box corresponds to increasing the sum of the weights by at most 1
(see Equation \eqref{eq:alcoveyoung}). Therefore each of these
submodules is in the alcove or on its boundary (as
$V^{\lambda_{a}^c}_ξ$ was in the alcove $C^<$).  Since typical modules in
the alcove or on its boundary remain simple after specialisation, we
obtain a set of simple modules which are pairwise non-isomorphic.

Moreover, from the Lie superalgebra decomposition we know that the sum of
their vector space dimensions adds up to the dimension of
$V^{\lambda_{a}^c}_ξ\otimes \mathsf{v}$.  Thus, by Lemma \ref{L:3} we have $V^{\lambda_{a}^c}_ξ\otimes \mathsf{v}$ is a direct sum
of these simple modules and we proved Fact \ref{F:Vlambdaa}.
   
To prove {\bf Fact \ref{F:TypRet}} we continue our discussion about
Young diagram corresponding to $V^{\lambda_a^c}_ξ$ be as above.  The
typical module $V^{\lambda^0_{a}}_ξ$ corresponds to the two empty
Young diagram and the complex number $a$.  Since taking tensor product
with $\mathsf{v}$ corresponds to a direct sum decompositions
determined by all possible ways of adding a single box, we see that we can
iterate this process and obtain the Young diagram corresponding to
$V^{\lambda_a^c}_ξ$.  Then, Fact \ref{F:Vlambdaa} implies that each
iteration stays of the form given in Equation \eqref{E:RTN} and the
result follows.
      
To prove {\bf Fact \ref{F:TensorVaVb}}, we consider $V\in \catpb_{\wb a}$, $V=T \oplus N$, as in Equation \eqref{E:RTN}. Let $V^{\lambda^0_b}_ξ\in \catpb_{\wb b}$ such that $\wb a+\wb b\notin \Xp $. We have:
 $$V^{\lambda^0_b}_ξ \otimes V= \left( V^{\lambda^0_b}_ξ \otimes T \right) \oplus \left(  V^{\lambda^0_b}_ξ \otimes N \right).$$
 Using Lemma \ref{L:NegIdeal}, we have that
 $V^{\lambda^0_b}_ξ \otimes N$ is negligible. Then, we need to prove
 that $V^{\lambda^0_b}_ξ \otimes T$ has the form from Equation
 \eqref{E:RTN}. Having in mind the form of semi-simple typical modules
 from Definition \ref{D:4}, it suffices to prove that
 $ V^{\lambda^0_b}_ξ\otimes V^{\lambda_{a}^c}_ξ $ has the form from
 Equation \eqref{E:RTN}, where $V^{\lambda_a^c}_ξ $ is any typical
 module in $\catp_{\wb a}$.  From Fact \ref{F:TypRet} we know
 $V^{\lambda_a^c}_ξ $ is a direct summand of
 \newcommand{\cg}{}
 $V^{\lambda^0_{a \cg +t}}_ξ \otimes W$ where $W=\mathsf{v}^{\otimes s}\otimes (\mathsf{v}^*)^{\otimes s'}\otimes \bar\C^{\otimes u}$ for some
 $s,s',t,u\in \N$.  
 
 Moreover, Lemma \ref {L:21} implies
 $V^{\lambda^0_b}_ξ \otimes V^{\lambda^0_{a\cg+t}}_ξ$ is semi-simple
 typical. Applying Fact \ref{F:Vlambdaa}, $s+s'$ times, we conclude that
 $V^{\lambda^0_b}_ξ \otimes V^{\lambda^0_{a\cg+t}}_ξ \otimes
 W$ has the form from Equation \eqref{E:RTN}:
\begin{equation}\label{eq:4}
  V^{\lambda^0_b}_ξ \otimes V^{\lambda^0_{a\cg+t}}_ξ\otimes W
  =T \oplus N .
\end{equation}
But $V^{\lambda^0_b}_ξ \otimes V^{\lambda_a^c}_ξ$ is a retract of the module in Equation \eqref{eq:4}. Then Lemma 5.2.2 of \cite{AG17}, implies there exist two retracts $T' \subseteq T$ and $N' \subseteq N$ such that 
$$V^{\lambda^0_b}_ξ\otimes V^{\lambda_a^c}_ξ=T' \oplus N'.$$
Since $N$ is negligible and $N'$ is a retract of it, Lemmas \ref{L:NegIdeal} implies that $N'$ is also negligible.
Thus,  $ V^{\lambda^0_b}_ξ\otimes V^{\lambda_a^c}_ξ$ has the desired form.  
\end{proof}
\subsection{Construction of $\catN$: the quotient of $\catpb$ by negligible morphisms}
\begin{definition}[Negligible morphisms]\
  \begin{enumerate}
  \item We say that a morphism $f:V\to W$ between two objects of 
    $\catpb$ is \emph{negligible} if 
    there exists a negligible object
    $X$ and morphisms $g: X \rightarrow W$ and $h: V\to X$ such that
    $f=g h$.
  \item Let $\Negl(V,W) $ be the sub-set of $\Hom_{\catpb}(V,W)$
    formed by all negligible morphisms from $V$ to $W$.
  \end{enumerate}
\end{definition}
 \begin{lemma} Let $V,V',W,W'\in\catpb$. Let $f\in \Negl(V,W)$, $k\in \Negl(W,V')$, $g\in \Hom_{\catpb}(V',V)$ and $h\in \Hom_{\catpb}(V',W')$.  Then
   \begin{enumerate}
   \item $\Negl(V,W)$ is a sub-vector space of $\Hom_{\catpb}(V,W)$, 
   \item $f \circ g\in  \Negl(V',W)$ and $g\circ k\in \Negl(W,V)$,
   \item $f\otimes h\in \Negl(V\otimes
     V',W\otimes W')$ and $ h\otimes f\in \Negl(V'\otimes
     V,W'\otimes W)$.  
   \end{enumerate}
 \end{lemma}
 \begin{proof}
Property (2) is immediate from the definition of negligible.  Property 
   (1) is true because if two morphisms $f,f'$ factors through the
   negligible objects $X,X'$ respectively, then any linear
   combination of $f$ and $f'$ factors through the negligible object
   $X\oplus X'$.  
   Finally, Property (3) follows because if $f$ factors through the negligible
   object $X$ a then $f\otimes h$ factors through the
   negligible object $X\otimes V'$ (similarly for left tensoring).
 \end{proof}

  We now describe a purification process of $\catpb$ which will produce a category $\catN$ where all negligible morphisms are zero. 
  
\begin{definition}[Category $\catN$ as purification of $\catpb$]
Let $\catN$ be the category whose objects are the same as the ones of $\catpb$ and 
whose set of morphisms between two objects $V$ and $W$  is
$$\Hom_{{\catN}}(V,W)=\Hom_{\catp}(V,W)/\Negl(V,W).$$
\end{definition}

There is an obvious functor $\mathcal{F}: \catpb \to \catN$ which is the identity on objects and maps a morphism to its class modulo negligible morphisms:
\begin{enumerate}
\item $\mathcal{F}(A)=A, \text{ for all } A \in Ob(\catpb)$,
\item $\mathcal{F}(f)=[f] \in \Hom_{\catN}(A,B), \text{ for all } f \in \Hom_{\catpb}(A,B)$. 
\end{enumerate}
We will use the functor $\mathcal{F}$ to induce structures from $ \catpb$ to analogous structure on $\catN$.  
\begin{lemma}\label{L:catNGraded}
The category $\catN$ is a $\Gp=\C/\Z$-graded ribbon $\C$-linear category  whose structures are 
 induced from the analogous structures of $\catpb$.  
\end{lemma} 
\begin{proof}
First, we will show that $\catN$ is a ribbon $\C$-linear category. 
Using the functor $\mathcal{F}$, we can induce the tensor $\C$-linear structure of $\catpb$ onto $\catN$.   We also define the dual structure on $\catN$ as the one coming  from $\catpb$, via the functor $\mathcal{F}$.
Since the braiding and duality morphisms in $\catpb$ satisfy the
compatibility conditions for a pivotal
structure, then the corresponding braiding and duality morphisms under $\mathcal{F}$ will also satisfy these compatibility conditions in $\catN$.  

Recall the definition of a $\Gr$-graded category given in Subsection \ref{D:Gstr}.  
From Theorem \ref{T:1}, we know that $\catpb$ is $\C/\Z$-grading.  For any $g\in \C/\Z$, define $$\catN_g:=\mathcal{F}(\catpb_g).$$  It is easy to see this gives a $\C/\Z$-grading on $\catN$.    
\end{proof}

\subsection{Relative pre-modular structure on $\catN$}
Let $d=\frac{|m-n|}{\gcd(m,n)}$.  Then one can check that $d.w_m\in\Lambda_R$ so that $d.\ell.w_m\in \Lambda_\Zt^0$.  Next, we define a free
  realization for our category by considering elementary modules with
  weights which are zero except the component corresponding to the
  generator $H_m$.
\begin{definition}[Free realisation]\label{N:2}
 Consider the commutative
 subgroup of $\Zt$: $$\Ztt=(\Z/2\times\Z,+)<\Zt=(\Z/2\times\Lambda_\Zt^0,+)$$ where the $\Z$
  factor of $\Ztt$ is mapped to $\Z.d\ell w_m<\Lambda_\Zt^0$.
  The free realization $\{\sigma(z)\}_{z\in \Zt}$ of $\Zt$ in $\catd$ induces a free realization $\{\sigma(z)={\mathcal F}(\sigma(z))\}_{z\in \Ztt} $ of $\Ztt$ in $\catN$.
\end{definition}

As above let $ \Xp:=  \left( \frac12 \Z \right) / \Z   \subseteq \Gp$.  
\begin{proposition}[Generic Semi-simplicity]\label{P:GenSS}
The category $\catN$ is generically semi-simple with finite set of regular simple objects:
$$\Tp(g):=\Big\{ \  V^{\lambda_a^c}_{\xi}\  \mid  \  a\in \C, \, c\in \N^{{\rk}-1},\, \lambda_a^c \in C^{<},\,0\leq Re(a) < d\ell,\,  \bar{a}=g \in \C/\Z \ \Big\}$$
for $g \in \Gp \setminus \Xp$.
\end{proposition}
\begin{proof}
This property comes from the structural description of generic modules from $\catpb$, given in Theorem \ref{T:1}, and the fact that we are working in the quotient category $\catN$. 

Let $g\in \C / \Z$ be a generic grading, $g \notin \Xp$. We denote the set of typical modules from the alcove corresponding to this grading by:
$$S_g:= \Big\{ V^{\lambda_a^c}_{\xi} \otimes \bar{\C}^{\otimes\gamma} \  \mid  \  a\in \C, \ c\in \N^{{\rk}-1}, \lambda_a^c \in C^{<}, \bar{a}=g \in \C/\Z , \gamma \in \{ 1,0 \} \Big\} .$$ 
Using Theorem \ref{T:1} and the property that the negligible objects
are not seen (up to isomorphism) in $\catN$, we notice that $S_g$ is a
completely reduced dominating set for $\catN_g$.
In order to arrive at
a finite set, let us consider $\Tp(g)$ given in the statement of the theorem.
Using the characterization of the alcove given in Equation  \eqref{D:7}, we see
in $\Tp(g)$ we have a finite number of possibilities for the
$c$-component of the weight
and $d\ell$ possible values for $a$. This shows that $\Tp(g)$ is a finite
set. Moreover, we notice that:
$$S_g=\Tp(g) \otimes \sigma(\Ztt).$$
We conclude that $\Tp(g)$ satisfies the requirements concerning the generic semi-simplicity condition for $\catN$. 
\end{proof}

Let $a\in \C$ with $\wb a \notin \Xp$.   Then Proposition \ref{P:GenSS} implies that $\catN_{\wb a}$ is semi-simple and so each object of $\catN_{\wb a}$ is projective.   
Thus, the ideal generated by $V^{\lambda^0_a}_ξ$ in $\catN$   is the ideal of projective objects $\Proj_{\aleph}$ in $\catN$, i.e.\  $\Proj_{\aleph} =\ideal_{V^{\lambda^0_a}_ξ}$.

\begin{lemma}[M-trace]
The family $\mt^{\aleph}_V:\End_{\catN}(V)\to\kk$ indexed by  $V\in \Proj_{\aleph}$ given by 
$
\mt^{\aleph}_V([f])= \mtp_V(f)
$
is an m-trace where $\mtp$ is the m-trace on $\idealp$ in $\catp$
given in Theorem  \ref{C:22}.   
\end{lemma}
\begin{proof}
  First, we need to show that $\mt^{\aleph}_V([f])$ does not depend on the
  choice of representative of $[f]$.  Suppose $f, f' \in \End_{\catp}(V)$ with $[f]=[f']$.  Then by definition there exist a negligible object $X$ and morphisms $g: X \rightarrow V$ and $h: V\to X$ such that $f=f'+gh$.  We have 
$$
\mt^{\aleph}_V([f])=\mtp_V(f'+gh)=\mtp_V(f')+\mtp_V(gh)=\mtp_V(f')+\mtp_X(hg)=\mtp_V(f')=\mt^{\aleph}_V([f']).
$$
Now to check that the partial trace property:  let $[f]\in \End_{\catp}(U\otimes V)$ then
$$
\mt^{\aleph}_{U\otimes V}([f])=\mtp_{U\otimes V}(f)=\mtp_U(\ptr_V(f))=\mt^{\aleph}_U([\ptr_V(f)])=\mt^{\aleph}_U(\ptr_V([f]))
$$
where the last equality comes from Lemma \ref{L:catNGraded}, as the pivotal structure of $\catN$ is induced from the pivotal structure of  $\catpb$.  The cyclicity property is proven in a similar way.  
\end{proof}

\begin{lemma}[Compatibility]
The bilinear map
$$\psi: \Gp \times \Ztt \rightarrow \kk^* \; \text{ given by 
}\; \psi(\wb a, k)=ξ^{\ell \wb{a}k\frac{2dmn}{n-m}}$$ satisfies the
compatibility property of Definition \ref{D:6} for the category
$\catN$.
\end{lemma}
\begin{proof}
  Since $\wb \C$ is transparent, it is enough to consider
  $\sigma=\sigma(0,k)\in\ve(\Ztt)$ for some $k\in\Z$.  Let $a\in\C$,
  $\wb a\in\C/\Z$ be its class mod $\Z$ and $\lambda\in\h^*$ be the
  weight of a weight vector $v$ of an object $V\in\catN_{\wb a}$.  In
  particular, since $V$ is perturbative $\lambda-aw_m\in\Lambda_W$.  As in Lemma \ref{P:compat} the double braiding will acts on $v\otimes \sigma$ by the scalar $ξ^{2{\brk{\lambda,kd\ell w_m}}}$.  Since $\brk{kd\ell w_m,\Lambda_W}\in\ell\Z$ and
  $\brk{w_m,w_m}=\frac{\brk{2\rho_1,2\rho_1}}{(m-n)^2}=\frac{mn}{n-m}$, we have
  $$ξ^{2{\brk{\lambda,kd\ell w_m}}}=ξ^{2{\brk{aw_m,kd\ell w_m}}}=ξ^{\ell ak\frac{2dmn}{n-m}},$$
  which define a scalar depending only of $(\wb a,k)$ because
  $\frac{2dmn}{n-m}\in\Z$.
\end{proof}

Combining Lemma \ref{L:catNGraded} with the results of this
subsection, we have the following theorem:
\begin{theorem}\label{THEOREM2b}
  The category $\catN$ is a pre-modular $\Gp$-category relative to $(\Ztt,\Xp)$.
\end{theorem}

We have the following two conjectures.
\begin{conjecture}\label{conj:non-degenerateb}
  The category $\catN$ is non-degenerate.
\end{conjecture}

\begin{conjecture}\label{conj:modularb}
  The category $\catN$ is a modular $\Gp$-category.
\end{conjecture}

As mentioned above, Conjecture \ref{conj:modularb} implies Conjecture
\ref{conj:non-degenerateb}.

\begin{theorem}\label{T:non-degenerateSL21b}
  Conjecture \ref{conj:non-degenerateb} is true for $\g=\sll(2|1)$. 
\end{theorem}

\begin{proof} Here we will use the notation and calculations given in
  \cite{GP1}.  Let $g \in \Gp\setminus \Xp$.  For $\sll(2|1)$ the
  weights only depend on two variables:
  $\lambda_{a}^{m} \in \N\times \C$.  The proof of Lemma 5.9 in
  \cite{CGP14} shows that an edge colored with a typical module can
  slide over a circle component colored with the Kirby color (this
  only requires a category to be relative pre-modular).  Consider the
  diagrammatic equation defining $Δ_+$ or $Δ_-$.  Add a strand colored
  with $W\in \catN_{h}$ (with $h\in\Gp\setminus \Xp$) on the left of
  each side of this equation.  Then using the sliding property, one
  can slide the strand colored with $W$ into the circle component then
  slide the strand colored with $V$ out of the circle component.  It
  follows that $Δ_\pm$ do not depend on the choice of the module
  $V\in \catN_{g}$ nor the choice of group element
  $g\in \Gp\setminus \Xp$.

To compute $Δ_\pm$ we follow an analogous computation done in Section 2.2 of \cite{CGP14}.  
Let $a\in\C$ with $\bar a=g$ and let $Δ_+$ and $Δ_-$ be the scalar defined by Definition
\ref{d:ndeg} with $V=V(\lambda_{a}^{0})$ and
$\Omega_g=\sum_{k,m=0}^{\ell-1}\qdp(V(\lambda_{a+k}^m))V(\lambda_{a+k}^m)$.  As explained above the $Δ_+$ does
not depends on $V\in \catN_{\wb a}$ nor on $a\in\C$ nor on $\bar a=g\in \Gp\setminus \Xp$.   We have
\begin{align*}
Δ_+&=\sum_{k,m}\theta_{V(\lambda_{a+k}^m)}\theta_{V(\lambda_{a}^0)}\operatorname{Hopf}(V(\lambda_{a+k}^m),V(\lambda_{a}^0)^*)\frac{\qdp(V(\lambda_{a+k}^m))}{\qdp(V(\lambda_{a}^0))}
\end{align*}
where $\operatorname{Hopf}(V(\lambda_{\alpha}^c),V(\lambda_{\alpha'}^{c'}))$ is the value of the renormalized Reshetikhin-Turaev invariant $F'$ of the positive
Hopf link colored by $V(\lambda_{\alpha}^c)$ and $V(\lambda_{\alpha'}^{c'})$.  This value is given in  \cite{GP1}:
$$\operatorname{Hopf}(V(\lambda_{\alpha}^c),V(\lambda_{\alpha'}^{c'}))=q^{-(2\alpha+c+1)(2\alpha'+c'+1)}\frac{\qn{(c+1)(c'+1)}_q}{\qn1_q}.$$
The values for modified dimension and twist are also computed in  \cite{GP1}:
$$\qdp(V(\lambda_{\alpha}^c))=\frac{\qn{c+1}_q}{\qn{1}_q\qn{\alpha}_q\qn{\alpha+c+1}_q},\qquad\theta_{V(\lambda_{\alpha}^c)}=q^{-2\alpha(\alpha+c+1)}.$$
Using these values and $V(\lambda_{a}^0)^*=V(\lambda_{-a-1}^0)$, we obtain
\begin{align*}
  Δ_+&=\sum_{k,m}ξ^{\alpha_{km}}
            \frac{\qn{(m+1)(1)}_ξ}{\qn1_ξ}
            \frac{\qn{m+1}_ξ\qn a_ξ\qn{a+1}_ξ} {\qn1_ξ\qn {a+k}_ξ\qn{a+k+m+1}_ξ}
\end{align*}
where
\begin{align*}
  \alpha_{km}&=\bp{-2(a+k)(a+k+m+1)}+\bp{-2a(a+1)}+\bp{-(2a+2k+m+1)(-2a-1)}\\
             &=-2{k}^{2}+1+m(1-2k).
\end{align*}
So we have
\begin{align}
  Δ_+&=\sum_{k}ξ^{-2{k}^{2}+1}\frac{\qn a_ξ\qn{a+1}_ξ}{\qn1_ξ^2\qn {a+k}_ξ}\sum_{m}ξ^{m(1-2k)}\frac{\qn{m+1}_ξ^2}{\qn{a+k+m+1}_ξ}.
\end{align}
Since $Δ_+$ does not depends on $a$, it is equal to the limit
$ξ^a\to\infty$.  Then $\qn{a+c}\sim ξ^{a+c}$ and we get
\begin{align*}
  Δ_+&=\sum_{k}ξ^{-2{k}^{2}+1}\frac{ξ^{2a+1}}{\qn1^2_ξξ^{a+k}}
            \sum_{m}ξ^{m(1-2k)}\frac{\qn{m+1}_ξ^2}{ξ^{a+k+m+1}}\\
          &=\sum_{k}ξ^{-2{k}^{2}-2k+1}\frac{1}{\qn1_ξ^2}
            \sum_{m}\left(ξ\cdot ξ^{(1-2k)m}-2ξ^{-2km}+ξ^{-1}ξ^{(-1-2k)m}\right).
\end{align*}
Since the last sum is a sum of roots of unity of the form
$\sum_mξ^{cm}$, it vanishes unless $\bar c=0\in\Z/\ell\Z$ in
which case its value is $\ell$.  Thus for each of the three terms
there is only one value of $k$ for which the sum over $m$ does not
vanish which happen only for respectively $k=\frac{\ell+1}2$, $k=0$ and $k=\frac{\ell-1}2$ respectively. So we have:
\begin{align*}
  Δ_+&=\frac{2\ell(ξ^{\frac{\ell+1}2}-ξ)}{\qn1_ξ^2}\neq0.
\end{align*}
Similarly,
\begin{align*}
  Δ_-&=\sum_{k}ξ^{2{k}^{2}-1}\frac{\qn a_ξ\qn{a+1}_ξ}{\qn1_ξ^2\qn {a+k}_ξ}\sum_{m}ξ^{m(2k-1)}\frac{\qn{m+1}_ξ^2}{\qn{a+k+m+1}_ξ}\\
  &=\sum_{k}ξ^{2{k}^{2}-2k-1}\sum_{m}ξ.ξ^{m(2k-1)}-2ξ^{m(2k-2)}+ξ^{-1}ξ^{m(2k-3)}\\
  &=\frac{2\ell(ξ^{\frac{\ell-1}2}-ξ^{-1})}{\qn1_ξ^2}=\overline{Δ_+}\neq0.
\end{align*}
\end{proof}

Theorem \ref{T:non-degenerateSL21} implies that the 3-manifold invariants of \cite{CGP14} can be associated to $\catN$ in the case of $\g=\sll(2|1)$.

\end{document}